\RequirePackage{fix-cm}
\documentclass[smallcondensed,referee,envcountsect,]{article}

\def\squareforqed{\hbox{\rlap{$\sqcap$}$\sqcup$}}
\def\qed{\ifmmode\else\unskip\quad\fi\squareforqed}

\def\smartqed{\def\qed{\ifmmode\squareforqed\else{\unskip\nobreak\hfil
\penalty50\hskip1em\null\nobreak\hfil\squareforqed
\parfillskip=0pt\finalhyphendemerits=0\endgraf}\fi}}
 \def\keywordname{{\bfseries Keywords}}%

\def\keywords#1{\par\addvspace\medskipamount{\rightskip=0pt plus1cm
\def\and{\ifhmode\unskip\nobreak\fi\ $\cdot$
}\noindent\keywordname\enspace\ignorespaces#1\par}}

\def\subclassname{{\bfseries Mathematics Subject Classification
(2000)}\enspace}
\def\subclass#1{\par\addvspace\medskipamount{\rightskip=0pt plus1cm
\def\and{\ifhmode\unskip\nobreak\fi\ $\cdot$
}\noindent\subclassname\ignorespaces#1\par}}

\smartqed  

\usepackage[title]{appendix}
\usepackage{graphicx, enumerate, url}
\usepackage{amssymb,amsmath,stmaryrd,setspace}
\usepackage{mathptmx}      

\usepackage{diagbox}
\usepackage{hyperref}
\usepackage{cleveref}

\newtheorem{theorem}{Theorem}
\newtheorem{proposition}{Proposition}
\newtheorem{lemma}{Lemma}

\newenvironment{proof}{\paragraph{Proof:}}{\itshape}{\rmfamily}

\numberwithin{equation}{section}

\newcommand{\R}{\mathbb{R}}
\newcommand{\ttF}{\mathtt{F}}

\newcommand{\ft}{\mathfrak{t}}
\newcommand{\sfT}{\mathsf{T}}
\newcommand{\sfH}{\mathsf{H}}

\newcommand{\OO}{\mathrm{O}}

\newcommand{\KK}{\mathbb{K}}
\newcommand{\C}{\mathbb{C}}

\newcommand{\AHerm}[1]{\mathrm{Skew}_{#1}}
\newcommand{\Flag}{\text{Flag}}
\newcommand{\St}[3]{\mathrm{St}_{#1,#2,#3}}  
\newcommand{\Gr}[3]{\mathrm{Gr}_{#1,#2,#3}}
\newcommand{\SSt}{\mathrm{St}}
\newcommand{\GGr}{\mathrm{Gr}}
\newcommand{\FFl}{\mathrm{Fg}}
\newcommand{\sfg}{\mathsf{g}}
\newcommand{\rM}{\mathrm{M}}
\newcommand{\tidY}{\tilde{Y}}
\newcommand{\UU}[2]{\mathrm{U}_{#1, #2}}
\newcommand{\lb}{\llbracket}
\newcommand{\rb}{\rrbracket}

\newcommand{\asym}[1]{\mathrm{skew}_{#1}}
\newcommand{\mrLambda}{\mathring{\Lambda}}
\newcommand{\mrY}{\mathring{Y}}
\newcommand{\mrZ}{\mathring{Z}}

\newcommand{\ccF}{\mathcal{F}}
\newcommand{\rmu}{\mathrm{u}}

\newcommand{\ba}{a}
\newcommand{\bb}{b}
\newcommand{\bd}{\mathbf{d}}
\newcommand{\bA}{A}

\newcommand{\bE}{E}
\newcommand{\bM}{M}
\newcommand{\bN}{N}
\newcommand{\bR}{R}
\newcommand{\bQ}{Q}

\newcommand{\bY}{Y}
\newcommand{\bZ}{Z}
\newcommand{\bS}{S}

\DeclareMathOperator{\ad}{ad}
\DeclareMathOperator{\diag}{diag}
\DeclareMathOperator{\Tr}{Tr}
\DeclareMathOperator{\Real}{Re}
\DeclareMathOperator{\TrR}{Tr_{\R}}
\DeclareMathOperator{\Exp}{Exp}
\DeclareMathOperator{\Hess}{Hess}
\DeclareMathOperator{\Log}{Log}
\DeclareMathOperator{\csr}{csr}
\DeclareMathOperator{\ssr}{ssr}
\DeclareMathOperator{\dI}{I}
\DeclareMathOperator{\frL}{L}

\begin{document}

\title{Closed-form geodesics and trust-region method to calculate Riemannian logarithms on Stiefel and its quotient manifolds
}


\author{Du Nguyen
}



\maketitle

\begin{abstract} We provide two closed-form geodesic formulas for a family of metrics on Stiefel manifold, parameterized by two positive numbers, having both the embedded and canonical metrics as special cases. The closed-form formulas allow us to compute geodesics by matrix exponential in reduced dimension for low-rank manifolds. Combining with the use of Fr{\'e}chet derivatives to compute the gradient of the square Frobenius distance between a geodesic ending point to a given point on the manifold, we show the logarithm map and geodesic distance between two endpoints on the manifold could be computed by {\it minimizing} this square distance by a {\it trust-region} solver. This leads to a new framework to compute the geodesic distance for manifolds with known geodesic formula but no closed-form logarithm map. We show the approach works well for Stiefel as well as flag manifolds. The logarithm map could be used to compute the Riemannian center of mass for these manifolds equipped with the above metrics. We also deduce simple trigonometric formulas for the Riemannian exponential and logarithm maps on the Grassmann manifold.
\subclass{65K10 \and 58C05 \and 49Q12 \and 53C25 \and 57Z20 \and 57Z25 \and 68T05 \and 68T45}
\end{abstract}

\section{Introduction}
\label{intro}
The Stiefel manifold, the manifold of orthogonal matrices of a given number of rows and columns, is important in various applications, including computer vision \cite{Turaga}, neural networks \cite{NishiAka}, statistics \cite{Baba}. It has two well-known metrics, the embedded metric and canonical metric defined in \cite{Edelman_1999}, which were treated separately in that foundational paper. While both have known exponential maps, allowing one to compute geodesic ending point when an initial position and velocity is given, they have no known closed-form logarithm map, which gives us a velocity vector with the minimum length that produces a geodesic connecting the initial point with an ending point. To compute statistics on Stiefel manifolds, for example, the Riemannian center of mass \cite{Karcher}, approximated geodesics are often used \cite{Baba}. There have been several research efforts to compute the logarithm map numerically, with \cite{Rentmee,Zimmer} treating the case of the canonical metric while \cite{Brynes,Sunda} treated the embedded metric. To our best knowledge, a common treatment of the logarithm map for both metrics is not yet available.

Recently, \cite{ExtCurveStiefel} introduced a family of metrics on the Stiefel manifold, parameterized by a real number that contains both the canonical and embedded metrics as special cases, with a closed-form geodesic formula. That geodesic formula requires calculating exponentials of matrices of dimension the longer size of a Stiefel matrix. Later, ignorance of \cite{ExtCurveStiefel}, in \cite{NguyenRiemann} we introduced essentially the same family of metrics, with a different parameterization. In this paper, we provide two new geodesics formulas for these metrics, generalizing both geodesics formulas in \cite{Edelman_1999}. In particular, the efficient formula for geodesics in Theorem 2.1 of \cite{Edelman_1999} generalizes to the whole family, including the embedded metric. The matrix exponentials required by our formulas depend on the shorter size of a Stiefel matrix, hence are efficient for low-rank matrices. A quick application of these formulas gives us new, trigonometric-like formulas for the exponential map and logarithm map of Grassmann manifolds (considered as a quotient manifold of Stiefel manifold), similar to those of the sphere.

For the Stiefel manifold, our generalized formulas for geodesics imply the search space for the logarithm map between given points could be identified with a Euclidean space with a substantially smaller dimension than the dimension of the manifold when matrices in the manifold are of low rank. We propose an approach to compute the logarithm map by {\it minimizing the square of the Frobenius distance between the ending point of a geodesic to a target point}. Thus, we treat the computation of the logarithm map as an {\it optimization problem with efficiently computed gradients}. In spirit, our approach is close to the shooting method of \cite{Brynes,SriKlass,Sunda}. The efficient evaluation of the gradient makes use of Fr{\'e}chet derivatives, (introduced with efficient computational procedures in \cite{Higham,Mathias,Havel}), bypassing the requirement for a partition of the time interval. The overall optimization problem is solved using a trust-region solver, with numerical Hessian (an L-BFGS-B solver is also likely to work).

The approach works well for all metrics in the family when the ending point of the geodesic is sufficiently close to the initial point, making this approach competitive for statistical problems, in particular the Riemannian center of mass \cite{Karcher}. For target points further from the initial point, the method is also robust, most of the time it returns a reasonable candidate for the logarithm map, generating a geodesic reaching the target point and has good potential for being length-minimizing. The objective function and its gradient for the canonical metric are simpler than those of other metrics in the family.

The application of Fr{\'e}chet derivatives to compute the gradient of the distance function could be generalized to other manifolds. To illustrate, we apply it to flag manifolds, introduced in the optimization literature in \cite{HelmkeMoore,Nishimori,YeLim}. From \cite{HelmkeMoore}, a flag manifold could be considered as the manifold of symmetric matrices with specified eigenvalues of given multiplicities (an {\it isospectral manifold}) and could be identified with a quotient space of a Stiefel manifold by a group of block-diagonal orthogonal matrices (the Grassmann manifold is a flag manifold with two eigenvalue blocks). Equipping a flag manifold with the quotient metric of a Stiefel metric in the family above, the flag manifold will have known geodesics but except for the Grassmann case, has no known logarithm map. Studying geodesics on a flag manifold is equivalent to studying geodesics on a Stiefel manifold, where the initial tangent direction satisfies a {\it horizontal condition}, while the ending point is required to reach not a target point but an element in its orbit. Our approach, minimizing the Euclidean distance in the symmetric matrix realization, with the help of Fr{\'e}chet derivatives, also works well in this case. For Grassmann manifolds, our numerical algorithm for the logarithm map matches the closed-form formula. For other flag manifolds, it also returns a reasonable candidate.

We will review basic notations in \cref{sec:notation}, prove the geodesic formulas for Stiefel manifolds in \cref{sec:stief_geo}, for Grassmann manifolds in \cref{sec:grassmann}. We review Fr{\'e}chet derivatives in \cref{sec:frechet} and present our algorithm for the logarithm map for Stiefel manifolds in \cref{sec:log_map_stief_all}. We review flag manifolds and report the results in \cref{sec:flag}. We finish with a calculation of the Riemannian center of mass, some notes on implementation, and a final discussion.
\section{Notations and background}\label{sec:notation}
Most of the results here are valid for both base fields $\KK=\R$ and $\KK=\C$, thus we denote $\KK$ to be either field if the results hold for both. We will denote $\KK^{n\times p}$ the space of matrices of $n$-rows and $p$-columns in $\KK$. We will denote the transpose $\ft$ to be the real transpose $\sfT$ if the base field $\KK$ is real, and the hermitian  transpose $\sfH$ if the base field $\KK$ is complex. If the base field is $\R$, the notation $\TrR A$ denotes the usual trace $\Tr A$ of a matrix $A$. If the base field is $\C$, $\TrR A$ denotes $\Real\Tr A$. The (real) inner product on a matrix space is $\Tr \ba\bb^{\sfT}$ if the base field is real, and $\TrR \ba\bb^{\sfH}$ if the base field is complex. We denote by $\|\|$ the usual Euclidean norm. We call a matrix $A$ $\ft$-antisymmetric if $A^{\ft} + A = 0$, $\ft$-symmetric if $A^{\ft} = A$. Define $\AHerm{\ft, p}$ to be the space of all $p\times p$ $\ft$-antisymmetric matrices, $\asym{\ft} A = \frac{1}{2}(A - A^{\ft})$ to be the anti-symmetrize operator. We denote the $\ft$-orthogonal group $\UU{\ft}{p}$ to be the group of matrices in $\KK^{p\times p}$ such that $U^{\ft}U = \dI_p$, thus $\UU{\sfT}{p}$ is $\OO(p)$, the real orthogonal group and $\UU{\sfH}{p}$ is the unitary group $\mathrm{U}(p)$.

On a Riemannian manifold $\rM$, a geodesic $Y(t)$ has the distance minimizing property and satisfies the geodesic equation of the form $\ddot{Y}(t) + \Gamma(\dot{Y}(t), \dot{Y}(t)) = 0$, where $\Gamma$ could be defined locally (by Christoffel symbols) or globally (by a Christoffel function, see \cite{Edelman_1999}). If $\tidY$ is a point on $\rM$ and $\eta$ is a tangent vector at $\tidY$, if $\bY(t)$ is a geodesic on $\rM$ with $\bY(0) = \tidY$ and $\dot{\bY}(0) = \eta$, we denote $\bY(1)$ by $\Exp_{\tidY}\eta$. Its inverse $\Log_{\tidY}\bZ$ denotes a vector $\eta$ of minimal length such that $\Exp_{\tidY}\eta=\bZ$. In general, $\Log$ may not exist, and may not be unique. However, as Stiefel and its quotient manifolds are compact, they are {\it complete}, and there is at least one geodesic connecting two points (the Hopf-Rinow theorem \cite{Gallier}, Theorem 16.17). We use the notation $\Exp^{\rM}$ to make explicit the manifold $\rM$ under consideration. Thus $\Exp^{\SSt}, \Exp^{\GGr}$ and $\Exp^{\FFl}$ denote the exponential maps on the Stiefel, Grassmann, and flag manifolds (see below). Similarly, we use the notation $\Log^{\rM}$ for the logarithm map.

We will denote by $\St{\KK}{p}{n}$ the Stiefel manifold over the field $\KK$, which consists of matrices $Y$ in $\KK^{n\times p}$ satisfying the relation $\bY^{\ft}\bY = \dI$. For $n=p$, $\St{\R}{n}{n}$ is $\OO(n)$, but we will focus on the component $\mathrm{SO}(n)$. The corresponding Grassmann manifold is denoted by $\Gr{\KK}{p}{n}$. Flag manifolds will be reviewed in \cref{sec:flag}. We use the notation $\lb Y \rb$ to denote the equivalent class of $Y$ (usually an element of the Stiefel manifold) in the quotient structure on the Grassmann or flag manifolds.

Fr{\'e}chet derivatives, the directional derivatives of a matrix function $f$ at point $\bA$ in direction $\bE$ will be denoted by $\frL_f(\bA, \bE)$. From the works \cite{Mathias,Havel} and especially \cite{Higham}, it could be computed very effectively. From our perspective, a calculation involving Fr{\'e}chet derivatives is as effective computationally as a closed-form solution, thus we can consider the gradient of the objective function discussed here computable in closed-form.

We will denote by $\csr(x)$ the entire function corresponding to $\cos x^{1/2}$, and $\ssr(x)$ the entire function corresponding to $x^{-1/2}\sin x^{1/2}$. These entire functions exist, as seen from their Taylor series expansions. They will be used to give analytic formulas for geodesics on Grassmann manifolds. We will denote by $\dI_{h, d}$ the $\R^{h\times d}$ matrix $\begin{bmatrix}\dI_{d\times d}\\ 0_{(h-d)\times d}\end{bmatrix}$ when $h\geq d$.
\section{Geodesics for Stiefel manifolds}\label{sec:stief_geo}
As noted, we will be working with matrix manifolds over a base field $\KK=\R$ or $\C$. Let $\St{\KK}{p}{n}\subset\KK^{n\times p}$ be the Stiefel manifold of matrices satisfying $Y^{\ft}Y = \dI_p$, let $\alpha_0, \alpha_1$ be two positive real numbers and $\bY\in \St{\KK}{p}{n}$. We define the operator-valued function
\begin{equation}\label{eq:stf_metric}
  \sfg\omega = \sfg_Y\omega := \alpha_0 \omega + (\alpha_1 - \alpha_0)\bY\bY^{\ft}\omega
\end{equation}
This gives us an inner product $\TrR\omega_1^{\ft}(\sfg_Y\omega_2)$ for $\omega_1, \omega_2$ in the ambient space $\KK^{n\times p}$, parametrized by $\bY\in\St{\KK}{p}{n}$. A tangent vector at $\bY$ could be considered as a matrix $\eta\in\KK^{n\times p}$ such that $\bY^\ft\eta$ is $\ft$-antisymmetric. The inner product induces a Riemannian metric on $\St{\KK}{p}{n}$. The case $\alpha_1 = \alpha_0$ corresponds to the embedded metric, the case $\alpha_1 = \frac{1}{2}\alpha_0$ corresponds to the canonical metric, both defined in \cite{Edelman_1999} for $\KK=\R$. Matrices of the form $\omega - YY^{\ft}\omega$ and $YY^{\ft}\omega$ are eigenvectors of $\sfg$ with eigenvalues $\alpha_0$ and $\alpha_1$, respectively. Thus $\sfg$ is positive definite.
The paper \cite{Edelman_1999} provided closed-form geodesics for the embedded and canonical metrics in separate formats. We generalize both results below:

\begin{theorem}\label{prop:st_geo} Let $\eta$ be a tangent vector to $\St{\KK}{p}{n}$ at a point $\tidY$. Let $\bA= \tidY^{\ft} \eta$. Let $\bQ$ be an orthogonal basis of the column span of $\eta - \tidY\tidY^{\ft}\eta$. Expressing $\eta - \tidY\tidY^{\ft}\eta = \bQ\bR$ in this basis. Let $\bS_0 := \eta^{\ft}\eta= -\bA^{2} + \bR^{\ft}\bR$ and $\alpha = \alpha_1/\alpha_0$, the geodesic equation for the metric $\sfg\omega = \alpha_0 \omega + (\alpha_1 - \alpha_0)\bY\bY^{\ft}\omega$
  in \cref{eq:stf_metric} is
  \begin{equation}  \label{eq:ageo}
    \ddot{\bY} +\bY\dot{\bY}^{\ft}\dot{\bY} + 2(1-\alpha)(\dI_n-\bY\bY^{\ft})(\dot{\bY}\dot{\bY}^{\ft})\bY=0\end{equation}
Each equation below describes the geodesic $\bY(t)$ with $\bY(0) = \tidY, \dot{\bY}(0) = \eta$:
  \begin{equation}
    \label{eq:st_geo1}
    \bY(t) = \begin{bmatrix} \tidY & \eta\end{bmatrix}\{\exp t\begin{bmatrix} (2\alpha-1)\bA & 2(\alpha-1)\bA^2-\bS_0 \\ \dI_p & \bA\end{bmatrix}\}\begin{bmatrix}\ \exp((1-2\alpha)t\bA) \\ 0\end{bmatrix}
  \end{equation}
  \begin{equation}
    \label{eq:st_geo2}
    \begin{gathered}
      \bY(t) = (\tidY \bM(t) + \bQ \bN(t))\exp((1-2\alpha)t\bA)
    \text{ with }\begin{bmatrix} \bM(t)\\ \bN(t) \end{bmatrix} = \exp t\begin{bmatrix}2\alpha \bA & -\bR^{\ft} \\ \bR & 0\end{bmatrix}\begin{bmatrix}\dI_p\\ 0\end{bmatrix}
      \end{gathered}
  \end{equation}
\end{theorem}
\begin{proof}
  The metric defined in equations (43) of \cite{ExtCurveStiefel} corresponds to $\alpha_0=2, \alpha_1 = \frac{1}{\alpha+1}$, where $\alpha$ defined in that paper corresponds to $\frac{1}{2\alpha}-1$ in our notations. Since geodesics are unchanged by metric scaling, \cref{eq:ageo} follows from the geodesic equation (64) of \cite{ExtCurveStiefel}. (The proof in \cite{ExtCurveStiefel} extends to the complex case). We provided an alternative derivation in \cite{NguyenRiemann}.

Section 2.2.2 of \cite{Edelman_1999} shows a derivation of \cref{eq:st_geo1} for the case $\alpha=1$ (on ideas of Ross Lippert). We show it extends with little change for all $\alpha$. From $\bQ\bR = (\dI-\tidY\tidY^{\ft})\eta$, we have $\bR^{\ft}\bR = \bR^{\ft}\bQ^{\ft}\bQ\bR = \eta^{\ft}(\dI-\tidY\tidY^{\ft})\eta = \bS_0 -(-\bA)(\bA)$, from here $\bS_0 = -\bA^2+\bR^{\ft}\bR$. Following \cite{Edelman_1999}, put $\bA(t) := \bY(t)^{\ft}\dot{\bY}(t)$ (we should use a different symbol for $\bA(t)$ to be clear, but we will show it is constant shortly), and $\bS(t) := \dot{\bY}^{\ft}\dot{\bY}$, hence $\bS(0) = \bS_0$. Then, $\dot{\bA}(t)=\dot{\bY}(t)^{\ft}\dot{\bY}(t) + \bY(t)^{\ft}\ddot{\bY}(t)$ expands to (dropping $t$)
$$\begin{gathered}
\dot{\bA}(t)=\dot{\bY}^{\ft}\dot{\bY} - \bY^{\ft}\bY\dot{\bY}^{\ft}\dot{\bY} -2(1-\alpha) \bY^{\ft}(\dI_n-\bY\bY^{\ft})(\dot{\bY}\dot{\bY}^{\ft})\bY = 0
\end{gathered}
$$
where we have used the geodesic equation to expand $\ddot{\bY}$, and use $\bY^{\ft}\bY = \dI$ to reduce $\bY^{\ft}(\dI_n-\bY\bY^{\ft})$ to zero in the last equality. Thus, $\bA(t)$ is constant, and by the initial condition $\bA(t) = \bY^{\ft}\eta =\bA$ is $\ft$-antisymmetric, $\dot{\bY}^{\ft}(t)\bY(t) = -\bY^{\ft}(t)\dot{\bY}(t) = - \bA$. Next, expand $\dot{\bS}(t) = \ddot{\bY}(t)^{\ft}\dot{\bY}(t) + \dot{\bY}(t)^{\ft}\ddot{\bY}(t)$. We simplify $\ddot{\bY}(t)^{\ft}\dot{\bY}(t)$ using \cref{eq:ageo}:
$$\begin{gathered}
\dot{\bY}(t)^{\ft}\ddot{\bY}(t)=
-\dot{\bY}(t)^{\ft}\bY(t)\dot{\bY}(t)^{\ft}\dot{\bY}(t) - 2(1-\alpha)\dot{\bY}(t)^{\ft}(\dI_n-\bY\bY^{\ft}(t))(\dot{\bY}(t)\dot{\bY}(t)^{\ft})\bY(t) \\
=AS + 2(1-\alpha)SA + 2(1-\alpha)(-A)A(-A) =AS + 2(1-\alpha)SA + 2(1-\alpha)A^3
\end{gathered}$$
Thus $\dot{S}(t) = \dot{\bY}(t)^{\ft}\ddot{\bY}(t) + (\dot{\bY}(t)^{\ft}\ddot{\bY}(t))^{\ft}=(2\alpha-1)(AS - SA)$, where we have used the fact that $(A^3)^{\ft} + A^3 =0$. Therefore
$ S(t) = e^{(2\alpha-1)tA}S_0e^{(1-2\alpha)tA}$. Equation \ref{eq:ageo} becomes:
$$
\ddot{\bY} + \bY e^{(2\alpha-1)tA}\bS_0e^{(1-2\alpha)tA}+2(1-\alpha)(-\dot{\bY}A+\bY \bA^2) = 0$$
\begin{equation}
    \label{eq:ageo_l}
    \ddot{\bY}e^{(2\alpha-1)tA} + \bY e^{(2\alpha-1)tA}\bS_0+2(1-\alpha)(-\dot{\bY}A+\bY \bA^2)e^{(2\alpha-1)tA} = 0\end{equation}
    The last equation implies:
    $$\begin{gathered}
\frac{d}{dt}(\bY e^{(2\alpha-1)At}, \dot{\bY}e^{(2\alpha-1)tA}) =\\
(\bY e^{(2\alpha-1)tA}, \dot{\bY}e^{(2\alpha-1)tA})
\begin{bmatrix} (2\alpha-1)A & -\bS_0 +(2\alpha-2)A^2 \\
I &  A\end{bmatrix}
\end{gathered}
$$
    From here, \cref{eq:st_geo1} follows. Theorem 2.1 of \cite{Edelman_1999} is  a special case of \cref{eq:st_geo2} for $\alpha=\frac{1}{2}$. To prove it, set $Z(t) = \bY(t)e^{t(2\alpha-1)A}$, then take time derivatives:
$$\dot{\bY}(t)\exp (2\alpha-1)tA = \dot{Z}(t) -  (2\alpha-1)Z(t)A$$
$$\ddot{\bY}(t)\exp (2\alpha-1)tA  = \ddot{Z}(t)-2(2\alpha-1)\dot{Z}(t)A+(2\alpha-1)^2ZA^2$$
With $\bS_0 = \eta^{\ft}\eta = - A^2 +R^{\ft}R$, the left-hand-side of \cref{eq:ageo_l} becomes:
$$\begin{gathered}\ddot{Z}-2(2\alpha-1)\dot{Z}(t)A+Z(2\alpha-1)^2A^2 +Z(-A^2+R^{\ft}R) +\\
  2(1-\alpha)(-(\dot{Z}(t) - Z(t)(2\alpha-1)A)A + ZA^2)=0\end{gathered}$$
We collect the coefficients for $\dot{Z}$ as $-2(2\alpha-1) - 2(1-\alpha) = -4\alpha +2 -2 +2\alpha = -2\alpha$ and for $Z$ as $(2\alpha-1)^2A^2+(-A^2+R^{\ft}R) + 2(1-\alpha)(2\alpha)A^2 = R^{\ft}R$, hence the equation for $Z$ is
$$\ddot{Z} -2\alpha \dot{Z} +R^{\ft}R=0$$
From \cref{eq:st_geo1}, $Z(t) = \tidY M_1(t) + \eta M_2(t)$ for some function $M_1, M_2$, so
$$Z(t) = \tidY(M_1(t) + \tidY^{\ft}\eta M_2) + (I - \tidY\tidY^{\ft})\eta M_2(t)=\tidY M(t) + QR M_2 = \tidY M(t) + QN(t)$$
By linearity, both $M(t) = \tidY^{\ft}Z(t)$ and $N(t) = Q^{\ft}Z(t)$ satisfy the same equation as $Z$. The initial conditions are $M(0) = \dI_p, N(0) = 0, \dot{M}(0) = A, \dot{N}(0) = R$, so \cref{eq:st_geo2} gives us the solution. If $\eta-\tidY\tidY^{\ft}\eta$ is of full rank, the thin $QR$ decomposition satisfies the condition of the theorem as $Q = (\eta-\tidY\tidY^{\ft}\eta)R^{-1}$. In the degenerate case, we use an SVD decomposition to determine the rank, and $\bQ$ has fewer columns than $\bY$ does.\qed
\end{proof}
In our notations, the geodesic formula in equation (75) of \cite{ExtCurveStiefel} could be written as
\begin{equation}
  Y(t) = \exp(t((2\alpha-2)\tidY\tidY^\sfT\eta\tidY^\ft -\tidY\eta^\ft +\eta\tidY^\ft))\exp(t(1-2\alpha)\tidY\tidY^{\ft}\eta \tidY^{\ft})\tidY
\end{equation}
Without the help of the geodesic equation, it is not apparent that this formula is equivalent to \cref{eq:st_geo1} and \cref{eq:st_geo2}. As noted, it requires computing exponentials of two $n\times n$ matrices.
\section{Grassmann manifolds and geodesics}\label{sec:grassmann}
On the unit sphere $S^n\subset \R^{n+1}$, the exponential map $\Exp_{\tilde{x}}\eta$ at $\tilde{x}\in S^n$ is given by the simple formula $x_1 = \cos\|\eta\|\tilde{x} + \frac{\sin \|\eta\|}{\|\eta\|}\eta$ (which reflects the fact that the angle between $x_1$ and $\tilde{x}$ is $\|\eta\|$). This formula could be written as $x_1 = \csr(\eta^{\sfT}\eta)\tilde{x} + \ssr(\eta^{\sfT}\eta)\eta$, with $\csr z = \cos z^{1/2}$ and $\ssr z= z^{-1/2} \sin z^{1/2}$, understood as entire functions as mentioned in \cref{sec:notation}. This formula implies $\cos \|\eta\| = \tilde{x}^{\ft}x_1$, from here the logarithm map for the sphere is given by $\Log_{\tilde{x}}x_1= (x_1 - \tilde{x}(\tilde{x}^{\ft}x_1))(1-(\tilde{x}^{\ft}x_1)^2)^{-1/2}\arccos(\tilde{x}^{\ft}x_1)$.

As a corollary of \cref{prop:st_geo}, we provide similar formulas for horizontal geodesics and logarithm on the Grassmann manifold $\Gr{\KK}{p}{n}$, considered as a quotient of $\St{\KK}{p}{n}$. For other approaches to Grassmann geodesics, see \cite{BATZIES_2015,BAZ}.

Recall from \cite{Edelman_1999}, the Grassmann manifold $\Gr{\KK}{p}{n}$ could be considered as a quotient of the Stiefel manifold $\St{\KK}{p}{n}$ under the action of the $\ft$-orthogonal group $\UU{\ft}{p}$ by right multiplication $(Y, U)\mapsto YU$ for $Y\in\St{\KK}{p}{n}, U\in \UU{\ft}{p}$. We represent the equivalent class of $Y$ by $\lb Y \rb$, and the image of a tangent vector $\eta$ under this quotient as $\lb \eta\rb$. A tangent vector $\eta$ on the Stiefel manifold is horizontal with respect to this action if $\bY^{\ft}\eta = 0$. Since $\csr z$ and $\ssr z$ are entire functions, $\csr X$ and $\ssr X$ are defined for any square matrix $X$. Also, if we define $(1 - x^2)^{-1/2}\arccos x$ to be $1$ at $x=1$, we can add $x=1$ to its domain to get a function on the interval $[0, 1]$, which is right-continuous at $1$. With this convention, we can define $(\dI_p - \Sigma^2)^{-1/2}\arccos \Sigma $ for a diagonal matrix $\Sigma$ with diagonal entries in $[0, 1]$.
\begin{proposition}\label{prop:grass} Let $\csr z$ and $\ssr z$ be the analytic continuations of $\cos z^{1/2}$  and $z^{-1/2} \sin z^{1/2}$. At a point $\tidY$ on the Stiefel manifold $\St{\KK}{p}{n}$, for a tangent vector $\eta$ such that $\tidY^{\ft}\eta = 0$, the Stiefel exponential map is given by
  \begin{equation}\label{eq:exp_grass}
\Exp^{\SSt}_{\tidY}\eta = \tidY \csr \eta^{\ft}\eta + \eta\ssr \eta^{\ft}\eta = \tidY \cos (\eta^{\ft}\eta)^{1/2} + \eta\{(\eta^{\ft}\eta)^{-1/2}\sin (\eta^{\ft}\eta)^{1/2}\}
  \end{equation}
The class $\lb \Exp^{\SSt}_{\tidY}\eta\rb$ represents the Grassmann exponential $\Exp^{\GGr}_{\lb \tidY\rb}\lb\eta\rb$. Conversely, for $\tidY$ and $\bZ$ in $\St{\KK}{p}{n}$, let $\tidY^{\ft} \bZ =  U\Sigma V$ be an SVD decomposition of $\tidY^{\ft} \bZ\in\KK^{p\times p}$. Set
  \begin{equation}\label{eq:exp_log}  \eta = (\bZ - \tidY \tidY^{\ft}\bZ)V^{\ft} (\dI_p - \Sigma^2)^{-1/2}\arccos \Sigma  U^{\ft}
\end{equation}    
  then $\eta$ is a horizontal vector on the Stiefel manifold and $\Exp^{\SSt}_{\tidY}\eta = \bZ V^{\ft}U^{\ft}$. Thus, $\Exp^{\SSt}_{\tidY}\eta$ and $\bZ$ represent the same element in the Grassmann manifold $\Gr{\KK}{p}{n}$. The length of $\eta$ is $(\alpha_0\TrR\arccos^2\Sigma)^{1/2}$, which is the geodesic distance between the equivalent classes $\lb \tidY\rb$ and $\lb Z\rb$ of $\tidY$ and $Z$, hence $\eta$ is the lift of a geodesic logarithm $\Log^{\GGr}_{\lb \tidY\rb}\lb \bZ \rb$ of $\Gr{\KK}{p}{n}$ to $\tidY$.
\end{proposition}
\begin{proof} For a horizontal vector, $A=\tidY^{\ft}\eta=0$, and expand $\exp t\begin{bmatrix} 0 & -S\\ I & 0\end{bmatrix}$ in \cref{eq:st_geo1}
$$
  \begin{gathered}
    Y(t) =
      \begin{bmatrix} \tilde{Y} & \eta\end{bmatrix}\{\sum_{n=0}^{\infty} \frac{(-1)^n}{(2n)!}t^{2n}\begin{bmatrix} S^n & 0\\ 0 & S^n \end{bmatrix}
        + \sum_{n=0}^{\infty}\frac{(-1)^n}{(2n+1)!}t^{2n+1}\begin{bmatrix} 0 & -S^{n+1}\\ S^n & 0\end{bmatrix}  \}\begin{bmatrix}\ \dI_p \\ 0\end{bmatrix} =\\
      \begin{bmatrix} \tilde{Y} & \eta\end{bmatrix}\{\sum_{n=0}^{\infty} \frac{(-1)^n}{(2n)!}t^{2n}\begin{bmatrix} S^n \\ 0  \end{bmatrix}
        + \sum_{n=0}^{\infty}\frac{(-1)^n}{(2n+1)!}t^{2n+1}\begin{bmatrix} 0 \\ S^n\end{bmatrix}  \}
          =\\    \tilde{Y}\sum_{n=0}^{\infty} \frac{(-1)^n}{(2n)!}t^{2n} S^n 
          + \eta\sum_{n=0}^{\infty}\frac{(-1)^n}{(2n+1)!}t^{2n+1}  S^n
  \end{gathered}$$
  where $S = \eta^{\ft}\eta$, which gives us \cref{eq:exp_grass}. Now, assuming $\eta_1$ is horizontal such that $\Exp^{\SSt}_{\tidY}\eta_1  = \bZ W$ with $W^{\ft} W = \dI_p$, ($\Exp^{\SSt}_{\tidY}\eta_1$ connects $\tidY$ with an element in the orbit of $Z$). Multiply both sides of \cref{eq:exp_grass} by $\tidY^{\ft}$ and use the horizontal condition, we have  $\cos(\eta_1^{\ft}\eta_1)^{1/2} = \tidY^{\ft} \bZ W$, thus $\tidY^{\ft} \bZ = \cos(\eta_1^{\ft}\eta_1)^{1/2} W^{\ft}$ is a polar decomposition. This implies $\cos^2\{(\eta_1^{\ft}\eta_1)^{1/2}\} = \tidY^{\ft} \bZ \bZ^{\ft} \tidY = U\Sigma^2 U^\ft$ if $\tidY^{\ft} \bZ = U\Sigma V$ is an SVD decomposition. From here, eigenvalues of $(\eta_1^{\ft}\eta_1)^{1/2}$ must be diagonal entries of $\arccos\Sigma$ or $2\pi-\arccos\Sigma$, plus multiples of $2\pi$, with $\arccos\Sigma$ is the smallest possible choice. Hence, $(\alpha_0\TrR\arccos^2\Sigma)^{1/2}$ is a lower-bound of geodesic lengths which we will show is attainable. The analysis also shows if $\tidY^{\ft}Z$ is of full rank, the polar decomposition is unique, thus both the alignment matrix $W$ and $\csr\eta^{\ft}\eta$ are unique. In that case, if the minimal length $(\alpha_0\TrR\arccos^2\Sigma)^{1/2}$ is attainable, all entries of $\arccos\Sigma$ are in $(0, \pi/2)$, so $\ssr\eta^{\ft}\eta$ is invertible and we can solve for at most one $\eta$ from \cref{eq:exp_grass}, thus, length-minimizing geodesics between $\tidY$ and $Z$, if exist, must be unique. If an eigenvalue of $\arccos\Sigma$ is $\pi/2$ we may have more than one length-minimizing geodesics.
  
  Let us prove \cref{eq:exp_log} gives us a length minimizing geodesic. Direct substitution shows $\tidY^{\ft}\eta = 0$. Since $\Sigma$ is diagonal, elementwise functions of $\Sigma$ commute, together with $(\bZ-\tidY\tidY^{\ft}\bZ)^{\ft}(\bZ-\tidY\tidY^{\ft}\bZ) = \dI_p - \bZ^{\ft}\tidY \tidY^{\ft}\bZ$, \cref{eq:exp_log} implies
  $$\begin{gathered}
    \eta^{\ft}\eta  = U (\dI_p - \Sigma^2)^{-1/2}\arccos \Sigma  V(\dI_p - \bZ^{\ft}\tidY \tidY^{\ft}\bZ)V^{\ft} (\dI_p - \Sigma^2)^{-1/2}\arccos \Sigma  U^{\ft}\\
    = U(\dI_p-\Sigma^2)^{-1}\arccos^2 \Sigma U^{\ft} - U (\dI_p - \Sigma^2)^{-1/2}\arccos \Sigma  V\bZ^{\ft}\tidY \tidY^{\ft}\bZ V^{\ft} (\dI_p - \Sigma^2)^{-1/2}\arccos \Sigma  U^{\ft}
  \end{gathered}$$
  Using $\tidY^{\ft} \bZ = U\Sigma V$, hence $\bZ^{\ft}\tidY \tidY^{\ft}\bZ = V^{\ft}\Sigma^2 V$, the second term is
  $$\begin{gathered} 
    U(\dI_p - \Sigma^2)^{-1/2}(\arccos \Sigma)  VV^{\ft}\Sigma^2 VV^{\ft}  (\dI_p - \Sigma^2)^{-1/2}(\arccos \Sigma) U^{\ft} = U(\dI_p - \Sigma^2)^{-1}\Sigma^2 \arccos^2 \Sigma U^{\ft}
\end{gathered}
    $$
  Thus $\eta^{\ft}\eta = 
U(\dI_p - \Sigma^2)^{-1}(\dI_p - \Sigma^2)\arccos^2\Sigma U^{\ft} = U\arccos^2\Sigma U^{\ft}$, so $\arccos\Sigma = U^{\ft}(\eta^{\ft}\eta)^{1/2} U$ and $(\dI_p-\Sigma^2)^{1/2} = U^{\ft}\sin (\eta^{\ft}\eta)^{1/2} U$. Substitute these expressions to \cref{eq:exp_grass} to evaluate:
  $$\begin{gathered}Y(1) =\tidY U\Sigma U^{\ft} + (\bZ - \tidY \tidY^{\ft}\bZ)V^{\ft} (\dI_p - \Sigma^2)^{-1/2}\arccos \Sigma  U^{\ft} (\eta^{\ft}\eta)^{-1/2}\sin (\eta^{\ft}\eta)^{1/2} =\\
    \tidY U\Sigma U^{\ft} + (\bZ - \tidY \tidY^{\ft}\bZ)V^{\ft} U^{\ft} (\sin (\eta^{\ft}\eta)^{1/2})^{-1}U \arccos \Sigma U^{\ft}(\eta^{\ft}\eta)^{-1/2}\sin (\eta^{\ft}\eta)^{1/2}=\\ \tidY U \Sigma VV^{\ft}U^{\ft} + (\bZ - \tidY \tidY^{\ft}\bZ)V^{\ft} U^{\ft}
  \end{gathered}
  $$
The last expression is reduced to $\tidY \tidY^{\ft}Z  V^{\ft}U^{\ft} + (\bZ - \tidY \tidY^{\ft}\bZ)V^{\ft} U^{\ft} = \bZ V^{\ft} U^{\ft}$. Thus $Y(1)=\bZ V^{\ft} U^{\ft}$, representing the same Grassmann element as $\bZ$. The length of $\eta$ is $(\alpha_0\TrR\eta^{\ft}\eta)^{1/2}  = (\alpha_0\TrR\arccos^2\Sigma)^{1/2}$.
\qed
\end{proof}
We note for the case $p=1$, $\St{\R}{p}{n}$ is the real sphere, the Grassmann manifold is the projective space, $\Sigma = |\tidY^{\ft}\bZ|$. The geodesic above connects $\tidY$ to $\mathrm{sgn}(\tidY^{\ft}\bZ)\bZ$. In this case, \cref{eq:exp_log} is $(\bZ - \tidY \tidY^{\ft}\bZ)\mathrm{sgn}(\tidY^{\ft}Z) (1 - \Sigma^2)^{-1/2}\arccos|\tidY^{\ft}\bZ|$, versus the logarithm map of the sphere $(\bZ - \tidY \tidY^{\ft}\bZ) (1 - \Sigma^2)^{-1/2}\arccos(\tidY^{\ft}\bZ)$. Numerically, to evaluate \cref{eq:exp_grass} we can evaluate the entire functions on the eigenvalues of the symmetric matrix $\eta^{\ft}\eta$. To evaluate \cref{eq:exp_log} we only need one SVD decomposition of $\tidY^{\ft}Z$. We benefited from reading \cite{BAZ}.

\section{Fr{\'e}chet derivatives}\label{sec:frechet}
Recall \cite{Higham,Mathias,Havel} if $f(A)=\sum_{i=0}^{\infty}f_iA^i$ is a power series with scalar coefficient and $A$ is a square matrix, then the Fr{\'e}chet derivative $\frL_f(A, E) = \lim_{h\to 0}\frac{1}{h}(f(A+hE) - f(A))$ in direction $E$ could be expressed as
$$\frL_f(A, E) = \sum_{i=0}^{\infty} f_i \sum_{a+b=i-1}A^aEA^b$$
under standard convergence conditions, thus for entire functions it always exists. It is known that $\frL_f(A, E)$ and $f(A)$ could be computed together with a computational complexity of around three times the complexity of $f(A)$. There exist routines to compute Fr{\'e}chet derivatives of the exponential function in open source or commercial package, for example, the function {\it expm\_frechet} in SciPy \cite{scipy-nmeth}. To explain why we prefer Fr{\'e}chet derivatives, consider the simple case of the directional derivative of the exponential function. For a matrix function $X(t)$, the well-known formula $\frac{d}{dt} \exp X(t) = \exp X(t) \frac{1-\exp(-\ad_{X(t)})}{\ad_{X(t)}} \frac{dX(t)}{dt}$ (Section 3.2 of \cite{Gallier}) implies, for $X(t) =  A+tE$, $\frL_{\exp}(A, E) = \exp A\sum_{n=0}^{\infty} \frac{(-1)^n}{(n+1)\!!}\ad_A^n E$ (where $\ad_A E = AE - EA$). It is difficult to evaluate this formula directly. On the other hand, the library function {\it expm\_frechet} is effective computationally. The computational cost will be $O(p^3)$, where $p$ is the size of $A$ (it is also dependent on the logarithm of the eigenvalues of $A$). In the cases of symmetric and antisymmetric matrices, an eigenvalue decomposition or a Schur decomposition may improve the sparsity of the operations.

The following lemma (which is likely known, but we could not find a reference) shows Fr{\'e}chet derivatives are useful to compute matrix gradient of power series, 
\begin{lemma}\label{lem:frechet_grad} Let $f$ be an analytic function near $A$ and
 $\Tr (C \frL_f(A, E) D)$ is well-formed for some matrices $C, E, D$. Then
\begin{equation}\Tr\{ C \frL_f(A, E) D\} = \Tr\{ \frL_f(A, DC) E\}\end{equation}
\end{lemma}  
\begin{proof}From the analytic assumption we only need to prove this for $f(A) = A^n$. This follows from $\sum_{a+b=n-1}\Tr(CA^aEA^bD) = \sum_{b+a=n-1}\Tr(A^bDCA^aE)$.\qed
\end{proof}

\section{Computing the logarithm map for the Stiefel manifold}\label{sec:log_map_stief_all}
\subsection{Main algorithm}
\label{sec:log_map_stief}
We now turn to the problem of computing the logarithm map on the Stiefel manifold. From this section on we will work with the base field $\KK = \R$. We expect the complex case could be tackled similarly. We benefited from reading \cite{Zimmer,Rentmee,Brynes} while preparing this section.

For two elements $\tidY, Z\in\St{\R}{p}{n}$, our approach is to solve \cref{eq:st_geo2} for the tangent vector $\eta$, expressed as $\eta = \tidY A + QR$ in terms of the matrices $A=\tidY^{\sfT}\eta$ and $R=Q^{\sfT}\eta$. Let $\hat{A} = \begin{bmatrix}2\alpha A & -R^{\sfT}\\ R & 0\end{bmatrix}$. Equation (\ref{eq:st_geo2}) could be written in the following form
    \begin{equation}\label{eq:st_geo3}
      Y(t) = [\tidY Q](\exp t\hat{A} )\dI_{p+k, d}\exp ((1-2\alpha)tA)
    \end{equation}
We see  $Y(1)$ is in the column span of $\tidY$ and $Q$. Here, $k\leq p$ is the number of columns of $Q$. On the other hand, for any $(Q, A, R)$ with  compatible dimension such that $\tidY^{\ft}Q =0$ with $A^{\ft} + A = 0$ and $Q^{\ft}Q = \dI_{k}$ for an integer $k$, then $\tidY A + QR$ is a tangent vector and \cref{eq:st_geo3} describes a geodesic with $\dot{Y}(0) = \tidY A + QR$. We note the length of $\eta$ is $(\alpha_0\TrR R^{\ft}R-\alpha_1\TrR A^2)^{1/2}$.

We now show that, given $\tidY$ and $Z$, if $Q$ is such that $W = (\tidY|Q)$ is an orthogonal basis of the column span of $\tidY$ and $Z$, then there is a geodesic connecting $\tidY$ and $Z$ such that $\dot{Y}(0) =\eta$ is in the span of $W$. If $Y$ is in the intersection of the span of $W$ and $\St{\R}{p}{n}$, then  $Y= WV$ for $V\in \R^{(k+p)\times p}$. Next, $Y^\sfT Y = \dI_p$ implies $V^\sfT V = \dI_p$, thus, this intersection is itself a Stiefel manifold. The tangent vectors of this small Stiefel manifold are of the form $\eta = W\phi$, with $V^\sfT \phi$ is antisymmetric, and the induced metric is $\alpha_0\Tr\eta^{\sfT}\eta + (\alpha_1 - \alpha_0)\Tr\eta^{\sfT}YY^\sfT\eta = \alpha_0\Tr\phi^\sfT\phi + (\alpha_1 - \alpha_0)\Tr\phi^\sfT VV^\sfT\phi$. Since the intersection Stiefel manifold is complete, there is a geodesic in it, in the induced metric, connecting $\tidY$ and $Z$. From the expression of the induced metric, the geodesic in the small manifold is a geodesic in the full Stiefel manifold.

As the exponential map is injective near $0$, for $Z$ close enough to $\tidY$, the above analysis shows we can assume $Q$ is a complement basis of $\tidY$ in the column span of $\tidY$ and $Z$. Therefore, we will fix an orthogonal basis $Q$ of the column span of $Z - \tidY \tidY^{\sfT}Z$. We solve for $A$ and $R$ by minimizing the function $\ttF(A, R) = 1/2 \|Y_{(A, R)} - \bZ\|^2_F - p$ with $Y_{(A, R)} = Y(1)$ given by \cref{eq:st_geo2}, and $\|\|_F$ is the Frobenius norm. Since the manifold is complete, there is always a geodesic connecting $\tidY$ and $Z$, thus a minimum always exists with (global) minimal value $-p$. Our algorithm will be local, so it may return a value greater than $-p$, and even if the minimum value is $-p$, the geodesic found may not be of minimal length. However, if $\bZ$ is close enough to $\tilde{Y}$ (within its injectivity radius), the Riemannian logarithm is unique, and the $\eta$ from this approach gives us the logarithm map and geodesic distance. As we will see  in \cref{subsec:num_stiefel}, this objective function works quite adequately for our purpose. The main idea is the gradient of $\ttF$ could be computed by Fr{\'e}chet derivatives.

\begin{theorem}Let $\tidY$ and $\bZ$ be two points on the Stiefel manifold $\St{\R}{p}{n}$ and $Q\in \R^{n\times k}$ is an orthogonal basis of the column span of $\bZ-\tidY\tidY^{\sfT}\bZ$ ($k\leq p$ is its rank). For $(A, R)\in\AHerm{\sfT, p}\times \R^{p\times k}$, let $Y(t)$ be the geodesic starting at $\tidY$ with $\dot{Y}(0)=\eta := \tidY A + QR$ as in \cref{prop:st_geo}. Define the function $\ttF = \ttF(A, R) := (1/2)\Tr((Y(1) - \bZ)^{\sfT}(Y(1) - \bZ)-p$. With $\alpha:=\frac{\alpha_1}{\alpha_0}$, $\hat{A} := \begin{bmatrix}2\alpha A & -R^{\sfT} \\ R & 0\end{bmatrix}\in \R^{(p+k)\times (p+k)}$, $\dI_{p+k, p}:=\begin{bmatrix}\dI_p\\ 0_{k\times p}\end{bmatrix}$, we have
    \begin{equation}\label{eq:cost_stiefel}\ttF(A, R) = - \Tr \bZ^{\sfT}Y(1) = - \Tr \bZ^{\sfT}[\tidY  Q]\exp \hat{A} \dI_{p+k, p}\exp(1-2\alpha)A\end{equation}
Set $\mathring{L} := \frL_{\exp}((1-2\alpha)A, \bZ^{\sfT}[\tidY  Q](\exp\hat{A})\dI_{p+k, p})$, $\hat{L} = \begin{bmatrix}\hat{L}_{11} & \hat{L}_{12} \\ \hat{L}_{21} & \hat{L}_{22}\end{bmatrix} := \frL_{\exp}(\hat{A}, \dI_{p+k, p}\exp(1-2\alpha) A[\tidY Q])$
where the block structure of $\hat{L}$ is defined mirroring that of $\hat{A}$, then the gradient $\nabla \ttF(A, R)= (\nabla_A\ttF(A, R), \nabla_R\ttF(A, R))$ of $\ttF$ is given by
\begin{equation}\label{eq:grad_stief_log}
  \begin{gathered}
\nabla_A\ttF = (1-2\alpha)\asym{\sfT}\mathring{L} + 2\alpha \asym{\sfT}(\hat{L}_{11})\\
\nabla_R\ttF =  \hat{L}_{21} - \hat{L}_{12}^{\sfT}
\end{gathered}
\end{equation}      
\end{theorem}
\begin{proof} Expanding the formula for $\ttF(A, R)$, using the fact that $Y(1)^{\sfT}Y(1) = \dI_p, \bZ^{\sfT}\bZ = \dI_p$ we have $\ttF(A, R) = - \Tr \bZ^{\sfT}Y(1)$, then we can apply \cref{eq:st_geo3} to expand $\ttF(A, R)$. The directional derivative of $\ttF$ in the direction $(\Delta A, \Delta R)$ at $(A, R$), with $\Delta A\in \AHerm{\sfT}{p}$ and $\Delta R\in \R^{p\times k}$ is  
$$\begin{gathered}-\Tr \begin{bmatrix}\bZ^{\sfT}\tidY & \bZ^{\sfT}Q\end{bmatrix}\frL_{\exp}(\hat{A}, \begin{bmatrix}2\alpha \Delta A & -\Delta R^{\sfT} \\ \Delta R & 0 \end{bmatrix})\begin{bmatrix}\exp(1-2\alpha)A\\ 0_{k\times p}\end{bmatrix} -\\ (1-2\alpha)\Tr \bZ^{\sfT}[\tidY  Q](\exp\hat{A})\dI_{p+k, p}\frL_{\exp}(A, \Delta A)
    =\\ -
\Tr \frL_{\exp}(\hat{A}, \begin{bmatrix}\exp(1-2\alpha)A\\ 0_{k\times p}\end{bmatrix} \begin{bmatrix}\bZ^{\sfT}\tidY & \bZ^{\sfT}Q\end{bmatrix})\begin{bmatrix}2\alpha \Delta A & -\Delta R^{\sfT} \\ \Delta R & 0 \end{bmatrix}-\\ (1-2\alpha)\Tr \frL_{\exp}(A, \bZ^{\sfT}[\tidY  Q](\exp\hat{A})\dI_{p+k, p})\Delta A
  \end{gathered}
  $$
  Here, we applied matrix derivative rules and \cref{lem:frechet_grad}. The last expression simplifies to $-\Tr \{2\alpha\hat{L}_{11}\Delta A + \hat{L}_{12}\Delta R \} + \Tr \hat{L}_{21}\Delta R^{\sfT} - (1-2\alpha)\Tr\mathring{L}\Delta A$. As $\Delta A$ is antisymmetric, we need to anti-symmetrize the corresponding terms for the gradient, which gives us \cref{eq:grad_stief_log}.\qed
\end{proof}
Note that for the case of the canonical metric $\alpha=\frac{1}{2}$, we do not need to compute the $\exp(1-2\alpha)A$ in the expression for $\ttF(A, R)$ and $\mathring{L}$ does not appear in the gradient. $\hat{L}$ has the simpler form of $\frL_{\exp}(\hat{A}, \dI_{p+k, p}[\tidY Q])$ and is the only matrix needed for the gradient.

We can identify the search space of $A$ and $R$ with $\R^{p(p-1)/2+pk}$. To minimize $\ttF(A, R)$, we call a trust-region solver, specifically SciPy \cite{scipy-nmeth} {\it trust-krylov} solver in the scipy.optimize module, with the objective set to $\ttF$ defined by \cref{eq:cost_stiefel}, the Jacobian set to $\nabla\ttF (A, R)$ in \cref{eq:grad_stief_log}
and an approximate the Hessian product by $\frac{1}{h}(\nabla\ttF(A+h \xi_A, R + h\xi_R)-\nabla\ttF(A, R))$, where we use $h=10^{-8}$. Since the solver expects a vector, we wrap the initial guess, the variables for the objective function, the Jacobian, the Hessian as vectors, then unwrap these vectors inside the functions into pairs of matrices $(A, R)$'s to evaluate. Finally, we unwrap the solution vector to recover the optimal pair $(A, R)$ and the optimal vector $\eta = \tidY A + Q R$.

The initial tangent direction is the projection $\bZ - \frac{1}{2}(\tidY \bZ^{\sfT}\tidY + \tidY\tidY^{\sfT}\bZ)$. From here, the corresponding values for $A$ and $R$ are $A_0 = \frac{1}{2}(\tidY^{\sfT} \bZ - \bZ^{\sfT}\tidY)$ and $R_0 = Q^{\sfT}\bZ$. With these initial values, the geodesic logarithm on the Stiefel manifold with metrics parameterized by $(\alpha_0, \alpha_1)$ could be computed quite effectively.

{\bf For testing purposes, we will set $\alpha_0 = 1$ and set $\alpha_1=\alpha > 0$}. We have tested the routine for different values of $\alpha$, $n$ and $p$ and find the routine is quite robust. While we have not enforced the length minimizing condition, we find the logarithm, if found, mostly matches the generating tangent. For $\bZ$ close enough (for example $0.5\pi$, for $\alpha_0=1$) to the original $\tidY$, we almost always have convergence. Even with points as far as $1.5\pi$ we still often get convergence. The rate of convergence is dependent on the distance between $\tidY$ and $\bZ$. We will summarize the numerical results to show the method is robust for different values of $\alpha$, as well as dimensions of Stiefel manifold, then will analyze the Hessian in several scenarios to show the method is consistent with properties of the logarithm map.
\subsection{Numerical results of the Logarithm map for Stiefel manifolds and discussion}\label{subsec:num_stiefel}
\begin{table}
  \begin{tabular}{c c c c c c c}
\hline
dist ($\pi$)&$\alpha$&\#Iteration&\#Evals&\%Succ&\%Improve&\%Not worse\\
\hline
0.50 & 0.10 & 4.24/3.48 & 71.42/65.94 &  100/100 & 0/0 & 100/100\\
0.50 & 0.50 & 4.40/3.12 & 70.36/55.32 &  100/100 & 0/0 & 100/100\\
0.50 & 1.00 & 3.98/2.80 & 61.02/50.36 &  100/100 & 0/0 & 100/100\\
0.50 & 1.20 & 4.00/2.98 & 62.28/56.42 &  100/100 & 0/0 & 100/100\\
0.99 & 0.10 & 12.48/12.20 & 266.64/260.06 &  96/90 & 13/11 & 96/96\\
0.99 & 0.50 & 13.26/12.86 & 308.86/315.20 &  72/72 & 0/0 & 42/40\\
0.99 & 1.00 & 7.16/7.16 & 151.02/153.06 &  98/94 & 0/0 & 100/98\\
0.99 & 1.20 & 9.24/8.02 & 191.44/164.46 &  92/94 & 0/0 & 96/94\\
1.30 & 0.10 & 9.42/7.98 & 183.90/165.40 &  98/96 & 44/44 & 98/98\\
1.30 & 0.50 & 6.58/5.60 & 122.72/102.44 &  98/100 & 41/41 & 100/100\\
1.30 & 1.00 & 7.40/7.30 & 152.90/156.04 &  94/94 & 19/19 & 94/98\\
1.30 & 1.20 & 8.66/8.56 & 177.74/184.72 &  94/92 & 15/15 & 90/90\\
\hline
\end{tabular}
\caption{Logarithm calculation for different $\alpha$ and distance for $\St{\R}{2}{4}$.
The table compares two initializing schemes, $A_0 =0, R_0=0$ (left) and $A_0, R_0$ are from the projection of $Z-\tidY$ to the tangent space of $\tidY$ (right). Distance to the initial point is expressed in a multiple of $\pi$. Number of calculations in the trust-region algorithm is weighted, with weights for function, jacobian and Hessian evaluations are 1, 2, 4 respectively. The last two columns show if the length of the logarithm vector found is improved versus the generating length (a positive number on the \%Improve column), or worse (less than 100\% on the \%Not worse column).}
\label{tab:num_stiefel42}
\end{table}

\begin{table}
  \begin{tabular}{c c c c c c c}
\hline
dist ($\pi$)&$\alpha$&\#Iteration&\#Evals&\%Succ&\%Improve&\%Not worse\\
\hline
0.50 & 0.10 & 3.92/1.96 & 154.22/36.26 &  100/100 & 0/0 & 100/100\\
0.50 & 0.50 & 3.00/2.00 & 40.00/38.52 &  100/100 & 0/0 & 100/100\\
0.50 & 1.00 & 3.00/1.00 & 36.00/14.00 &  100/100 & 0/0 & 100/100\\
0.50 & 1.20 & 3.00/1.00 & 40.00/18.00 &  100/100 & 0/0 & 100/100\\
0.99 & 0.10 & 5.00/3.00 & 7638.56/686.96 &  100/100 & 0/0 & 100/100\\
0.99 & 0.50 & 4.00/3.00 & 55.00/52.00 &  100/100 & 0/0 & 100/100\\
0.99 & 1.00 & 3.00/1.00 & 36.00/14.00 &  100/100 & 0/0 & 100/100\\
0.99 & 1.20 & 4.00/2.00 & 51.00/29.00 &  100/100 & 0/0 & 100/100\\
1.30 & 0.10 & 5.00/3.00 & 86.00/56.00 &  100/100 & 0/0 & 100/100\\
1.30 & 0.50 & 5.00/3.00 & 74.00/52.00 &  100/100 & 0/0 & 100/100\\
1.30 & 1.00 & 4.00/2.00 & 47.00/48.04 &  100/100 & 0/0 & 100/100\\
1.30 & 1.20 & 5.00/2.00 & 19102.00/29.00 &  100/100 & 0/0 & 100/100\\
\hline
\end{tabular}
  \caption{Logarithm calculation for different alpha and distance for $\St{\R}{40}{1000}$on two initialization schemes}
\label{tab:num_stiefel100040}
\end{table}

We present in \cref{tab:num_stiefel42} and \cref{tab:num_stiefel100040} the results for $\St{\R}{2}{4}$ and $\St{\R}{40}{1000}$. For each distance in the ladder $(0.5\pi, 0.99\pi, 1.3\pi)$, we pick $\alpha_0 = 1$ and four values of $\alpha=\alpha_1/\alpha_0$ at $(0.1, 0.5, 1.0, 1.2)$. From a given initial point, we generate $50$ tangent vectors with length equal to the distance and apply the exponential map to create $50$ ending points for each distance and each $\alpha$, then call our logarithm routine implementing the algorithm in \cref{sec:log_map_stief}. Thus, we can compare the logarithm to the generating tangent vector. From the second column on, the tables tabulate the results of two choices of initial values, at $\eta_0 =0$, and at $\eta_0$ equal to the projection of $Z-\tidY$ to the tangent space of $\tidY$. The reported data include the number of trust-region iterations, a weighted number of evaluations, with function, jacobian and Hessian evaluations are weighted by $1, 2, 4$, respectively. We keep track of the percentage of success. Success indicates the solver return {\it success}, and the exponential map of the tangent vector found matches the target data point (note that in theory, the solver could return a local minimum of the objective function but the geodesic ending point may not match the target point, we do not count it as a success). We also report the number of times the length of the geodesic found is shorter than the length of the geodesic used to generate the ending point in the {\it \%Improve} column. The last column counts the percentage of the time when the length of the proposed geodesic exceeds the length of the generating vector. Except for one case of $\St{\R}{2}{4}$, there is a high percentage of the times we find a geodesic of comparable length to the generating tangent, and we could find geodesics with shorter lengths than the generating ones at further generating distance for $\St{\R}{2}{4}$.

One observation is the method seems to work better for thinner manifolds, where $p/n$ is smaller. For example, we could find at least a geodesic connecting the ending points at a distance of $1.3\pi$ for $\St{\R}{40}{1000}$, while for $\St{\R}{2}{4}$, we have a failure to converge even though the manifold is much smaller. The analysis in the next section indicates the difficulty may be with points where $\eta^{\sfT}\eta$ has more concentrated spectra, and when $\eta$ is generated randomly on a larger manifold, it has less chance of having concentrated spectra.

Convergence is slower for larger $p$, sometimes significantly slower. This may be because of the existence of highly non-convex regions. Besides, the dimension of the search space increases quadratically with $p$, which affects the number of CG steps in solving the trust-region subproblem, and matrix operations are of cost $O(p^3)$ or $O(np^2)$.

\subsection{Injectivity radius the spectra of the Hessian}\label{subsec:inject}
We now analyze the nonuniqueness of the Riemannian logarithm and the behavior of the Hessian of the objective function, which affects the rate of convergence.

The relevant Riemannian geometry concepts are cut-value and injectivity radius. Without going into details, we recall (Proposition 16.5, \cite{Gallier}) the exponential map for tangent vectors is injective within a certain radius near zero. For a unit tangent vector $\eta$, there may be a {\it cut value} $t_{c}(\eta)\in \R^{+}$ beyond it the distance from $\tidY$ to $\Exp_{\tidY} t\eta$ is smaller than $t$. The point $\Exp_{\tidY} t_c(\eta)\eta$ is then called a {\it cut point}. The injectivity radius at $\tidY$ is the least upper bound of the radius where $\Exp_{\tidY}$ is a diffeomorphism, or equivalently, the infimum of $\|t_c(\eta)\eta\|_{\tidY}$ over all directions.

For the case of the sphere ($p=1$), the cut value is $t_c(\eta) = \pi\sqrt{\alpha_0}$ for a unit tangent vector $\eta$, which expresses the simple fact that there are infinitely many geodesics connecting $\tidY$ to its antipodal point, while for any other point there is a unique length-minimizing  geodesic, the shorter of two great circle arcs connecting the points.

For $n > p \geq 2$, we now show the injectivity radius does not exceed $\pi\min(\sqrt{2\alpha_1}, \sqrt{\alpha_0})$. This follows from the fact that there exist closed geodesics (geodesics returning to the original point) of lengths $2\pi\sqrt{2\alpha_1}$ and $2\pi\sqrt{\alpha_0}$ from any point $\tidY$. Using \cref{eq:st_geo2}, such geodesics could be constructed with $A=0$, $Q \in \R^{n\times 1}$, $R\in \R^{1\times p}$ satisfies $(R^{\sfT}R)^{1/2} = 2\pi$, and $\eta = QR$. We have $\Exp_{\tidY}t\eta = \tidY\cos t(R^{\sfT}R)^{1/2}  + QR(R^{\sfT}R)^{-1/2}\sin t(R^{\sfT}R)^{1/2}$ is a closed geodesic, with length $t(\alpha_0R^{\sfT}R)^{1/2}= 2t\pi\sqrt{\alpha_0}$, which makes a closed circle at $t=1$ and a half circle at $t=1/2$, thus the exponential map cannot be a diffeomorphism for a radius greater than $\pi\sqrt{\alpha_0}$, as $-\eta/2$ and $\eta/2$ are mapped to the same point. Similarly, take $Q=0$ and $\eta = \tidY A$ with $A$ is an antisymmetric matrix with the top-left $2\times 2$ block is $A_0:=\begin{bmatrix}0 & -2\pi\\2\pi & 0\end{bmatrix}$ and other entries are zero. We have $Y(t) =\tidY\exp tA$ is a geodesic, and $Y(1)= \tidY$. The length of this geodesic is $\|\eta\|_{\tidY} = (\alpha_1\Tr(-A^2))^{1/2}= (\alpha_1 8\pi^2)^{1/2} = 2\pi\sqrt{2\alpha_1}$. There are two distinct geodesics connecting $\tidY$ to $\tidY\exp(1/2 A)$ (which are $Y_1(t) = \tidY\exp(t A)$ and $Y_2(t)=\tidY\exp(-tA)$).   Thus the exponential map is not a diffeomorphism for a radius greater than or equal to $\pi\sqrt{2\alpha_1}$ at any point. We note when $p=2$, $Y^{\ft}\eta$ is always proportional to $A_0$, and thus we have a higher chance of being near a closed geodesic, especially when $n$ is small. For the case $n=p$, the injectivity radius is bounded by $\pi\sqrt{2\alpha_1}$.
  
  Another estimate of the injectivity radius, from theorem 17.33 in \cite{Gallier}, uses an upper bound for sectional curvature of the Stiefel manifold, which is not yet available for all $\alpha$. For the canonical metric, \cite{Rentmee} showed the injectivity radius is bounded below by $\sqrt{4/5}\pi$.

  For the objective function, since $\ttF$ reaches a minimum when the geodesic ending point reaches the target, the Hessian operator product of $\ttF$ is a positive semidefinite operator at the corresponding optimal pair $(A, R)$. Empirically, we find its minimum eigenvalue eventually decreases when $Z$ moves away from $\tidY$. Denote by $\ttF(\tidY, Z, Q; A, R)$ the objective function corresponding to $\tidY, Z$ and $Q$ ($\tidY^\sfT Q =0$, $Q^\sfT Q = \dI_k$). $(A, R)$ will be the variables to minimize while $(\tidY, Z, Q)$ are parameters. If $\eta_*$ is a tangent vector and $t$ is a nonnegative number then the pair $(tA_*, tR_*)$ corresponding to $t\eta_*$ is a minimum of $\ttF(\tidY, \Exp_{\tidY} t\eta_*, Q; A, R)$, so the Hessian at $(tA_*, tR_*)$ is positive-semidefinite. We will show in \cref{sec:sf_zero_hess} that the Hessian is the identity map for $t=0$, and thus for small $t$ it will be positive-definite. As $t$ increases, we find the Hessian has decreasing eigenvalues and may be singular. Note that if $Z$ is a cut point in the differential geometric sense, there are at least two tangent directions from the $\tidY$ ending at $Z$, they are both minima of $\ttF$, but they could be isolated minima, hence their Hessian may still be non degenerated. However, we use the term {\it (empirical) approximate cut value} to indicate the value $t$ where the smallest Hessian value at a minimum point is under a threshold, (e.g. $10^{-8}$). Consider $\St{\R}{5}{10}$, using our formula for $\nabla\ttF$, we compute numerically the Hessian matrix $\Hess_{\ttF}(\tidY, Z, Q; A, R)$ by vectorizing increments $\Delta_A, \Delta_R$ over a basis of $\AHerm{\sfT}{p}\times \R^{p\times k}$.
We then run a search for the time where the smallest eigenvalue is under the threshold. The search stops when a time is found, within the time upper bound of $1.4\pi$. The graph on the left of \cref{fig:spectra} plots maximal and minimal time, over $100$ samples of $(\eta_*, Q)$ for each $\alpha$. The right-hand side shows the time evolution of the eigenvalues corresponding to one tangent vector. Numerically, we observe the Hessian degeneration typically happens beyond $\pi/2$, with higher $\alpha$ likely to result in higher (empirical approximate) cut values. Since the rate of convergence depends on the inverse of the Hessian eigenvalues, this also explains why convergence is slower for $Z$ further from $\tidY$.
\begin{figure}
\centering
\includegraphics[scale=0.4]{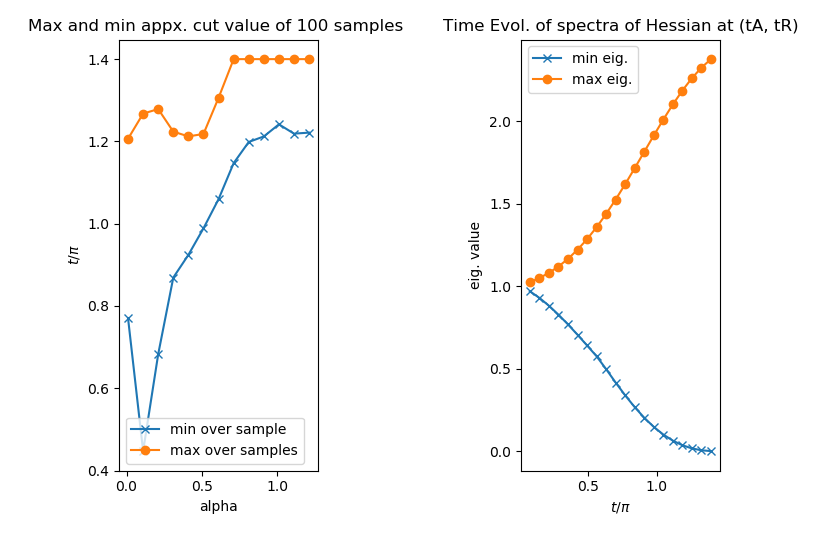}
\caption{Left, max and min approximate cut value, measured by the first time the Hessian becomes degenerated over $100$ sample directions (cutoff at $1.4\pi$) for $\St{\R}{5}{10}$. Right, the largest and smallest eigenvalues  of the Hessian of $\ttF$ along a randomly picked tangent direction at increasing distances.}
\label{fig:spectra}
\end{figure}
 \subsection{The Hessian at $A=0$ and $R=0$ }\label{sec:sf_zero_hess}
 Consider a pair of matrices $(\Delta A, \Delta R)\in \AHerm{\sfT, p}\times \R^{k\times p}$. Expanding $\ttF(\Delta A, \Delta R) = \ttF(\tidY, \bZ, Q; \Delta A, \Delta R)$ to power series in $\Delta A, \Delta R$ using \cref{eq:st_geo3}, the second-order terms are
$$\begin{gathered}-\Tr\bZ^{\sfT}\begin{bmatrix}\tidY & Q  \end{bmatrix}\{\frac{1}{2}\begin{bmatrix} 2\alpha\Delta A & -\Delta R^{\sfT}\\ \Delta R & 0_{k\times k}\end{bmatrix}^2 +
   \frac{1}{2}\dI_{p+k, p} \{(1-2\alpha)\Delta A\}^2\dI_{p+k, p}^{\sfT} +\\(1-2\alpha) \begin{bmatrix} 2\alpha\Delta A & -\Delta R^{\sfT}\\ \Delta R & 0_{k\times k}\end{bmatrix}\dI_{p+k, p} \Delta A\dI_{p+k, p}^{\sfT}\}\dI_{p+k, p} =\\
-\frac{1}{2}\Tr\bZ^{\sfT}\begin{bmatrix}\tidY & Q  \end{bmatrix}\begin{bmatrix} 4\alpha^2\Delta A^2 -\Delta R^{\sfT}\Delta R + (1-2\alpha)^{2}\Delta A^2 +4\alpha(1-2\alpha)\Delta A^2\\ 2\alpha \Delta R\Delta A + 2(1-2\alpha)\Delta R\Delta A  \end{bmatrix}
\\
  = -\frac{1}{2}\Tr\{ \bZ^{\sfT}\tidY(\Delta A^2 - \Delta R^{\sfT} \Delta R) + (2-2\alpha)\bZ^{\ft}Q\Delta R\Delta A\}
\end{gathered}
$$
If $\bZ = \Exp_{\tidY} \eta$ with $\|\eta\|$ small, then $\bZ^\sfT \tidY$ is close to $\dI_p$ while $\bZ^\sfT Q$ is close to zero, and Hessian bilinear form at $(A= \tidY^{\sfT}\eta, R=Q^{\sfT}\eta)$ is close to $\Tr\{-\Delta A^2 + \Delta R^{\sfT} \Delta R\}$, implying the Hessian operator is close to the identity operator. Thus, when $Z$ is close to $\tidY$, the Hessian at $(A, R)$ corresponding to a critical point will be close to the identity map, as discussed.

\section{Geodesics and logarithm for Flag manifolds}\label{sec:flag}
For our purpose, we will consider a flag manifold as a quotient of a Stiefel manifold by a group of block-diagonal orthogonal matrices. We compute the Euclidean distance between two flag elements using the realization in chapter 2 of \cite{HelmkeMoore}, as the manifold of matrices of fixed eigenvalues and multiplicities, reviewed below. A flag manifold is in a sense an intermediate object between a Stiefel manifold and its associated Grassmann manifold.

Again, we will assume $\KK=\R$. Let $n$ be a positive integer and assume we have a partition $\sum_{i=0}^q d_i = n$ with $d_0\geq 0$ and $d_i > 0$ if $i > 0$. This partition defines $q+1$ column blocks on an $n\times n$ matrix $\mrY$, with the $i$-th column block having size $d_i$. It also defines $q$ column blocks on an $n\times n_q$ matrix $Y$, with $n_q = n - d_0$, with the $i$-th column block ($1\leq i\leq q$) also having size $d_i$. Consider a tuple $\mathring{\lambda} = (\lambda_0, \lambda_1,  \cdots ,\lambda_q)$ of distinct real numbers and define the block diagonal matrix $\mrLambda = \diag(\lambda_0\dI_{d_0}, \lambda_1\dI_{d_1},\cdots,\lambda_q\dI_{d_q})$. The set of matrices of the form $\mrY\mrLambda \mrY^{\sfT}$, with $\mrY\in \OO(n)$ is a submanifold of $\R^{n\times n}$. From \cite{HelmkeMoore}, it could be identified with the quotient manifold $\OO(n)/(\OO(d_0)\times \OO(d_1)\times\cdots \OO(d_q))$, called the flag manifold $\Flag(n_1,\cdots,n_q; n)$, where $n_i = \sum_{j=1}^i d_i$. The idea of the identification of a flag manifold with a quotient of Stiefel manifold is the fact that when an eigenvalue of a symmetric matrix has multiplicities, the (normalized) eigenvectors are defined up to an orthonormal change of basis, as follow
\begin{lemma}\label{lem:flag_equiv} If $\mrLambda=\diag(\lambda_0\dI_{d_0}, \lambda_1\dI_{d_1},\cdots,\lambda_q\dI_{d_q})$ with $\lambda_i\neq \lambda_j$ if $i\neq j$, $\mrY$ and $\mrZ$ are two orthogonal matrices in $\OO(n)$ {\bf such that} $\mrY\mrLambda \mrY^{\sfT} = \mrZ\mrLambda \mrZ^{\sfT}$, where $\mrY$ and $\mrZ$ are expressed in block form $\mrY = [\bY_0, \bY_1, \cdots, \bY_q] = [Y_0 Y]$, $\mrZ = [Z_0, Z_1, \cdots, Z_q] = [Z_0 Z]$ with $Y_i\in \R^{n\times d_i}, Z_i\in\R^{n\times d_i}$, $0\leq i\leq q$, {\bf then} $Y_i^{\sfT}Z_j = 0$ if $i\neq j$; each $U_i=Y_i^{\sfT}Z_i\in \R^{n_i\times n_i}$ is an orthogonal matrix and $Z_i = Y_iU_i$.
  
Thus, we have $\mrZ = \mrY\diag(U_0, U_1\cdots, U_q)$. Conversely, if $\mrZ = \mrY\diag(U_0, U_1,\cdots, U_q)$ where each $U_i\in \OO(d_i)$ is an orthogonal matrix, then $\mrY\mrLambda \mrY^{\sfT} = \mrZ\mrLambda \mrZ^{\sfT}$.

If in the above, $\lambda_0 = 0$ and $\Lambda=\diag(\lambda_1\dI_{d_1},\cdots,\lambda_q\dI_{d_q})$ for distinct and non-zero $\lambda_i$'s, $1\leq i \leq q$, set $Y = [\bY_1, \cdots, \bY_q]$, then we have $\mrY\mrLambda \mrY^{\sfT} = Y\Lambda Y^{\sfT}$. We have $Y\Lambda Y^{\sfT} = Z\Lambda Z^{\sfT}$ for $Z = [Z_1, \cdots, Z_q]$ if and only if $Z = Y\diag(U_1\cdots U_q)$, and in that case, we have $U_i = Y_i^{\sfT}Z_i$.
\end{lemma}
\begin{proof}From $\mrY\mrLambda \mrY^{\sfT} = \mrZ\mrLambda \mrZ^{\sfT}$, multiply by $\mrZ^{\sfT}$ on the left and $\mrY$ on the right, we have
  $\mrZ^{\sfT}\mrY\mrLambda  = \mrLambda \mrZ^{\sfT}\mrY$ by the orthogonal assumption, which implies $\lambda_j Z_i^{\sfT}Y_j = \lambda_i Z_i^{\sfT}Y_j$, so $Z_i^{\sfT}Y_j= 0$ for $i\neq j$. Since $Z_i$ is orthogonal to all $Y_j$ for $j\neq i$, expressing $Z_i$ in the basis $\mrY$, we have $Z_i = Y_iU_i$ for $U_i\in \R^{d_i\times d_i}$, thus $U_i = Y_i^{\sfT}Z_i$. From $Z_i^{\sfT}Z_i = \dI_{d_i}$ we have $\dI_{d_i}=U_i^{\sfT}Y_i^{\sfT}Y_iU_i = U_i^{\sfT}U_i$. Conversely if each $U_i$ is orthogonal, $\mrZ\mrLambda\mrZ^{\sfT} = \mrY\diag(U_0, U_1,\cdots, U_q)\mrLambda\diag(U_0^{\sfT}, U_1^{\sfT},\cdots, U_q^{\sfT})\mrY^{\sfT}=\mrY\mrLambda\mrY^{\sfT}$. The remaining part is just the special case when $\lambda_0 = 0$.\qed
\end{proof}
We will pick $\lambda_0=0$ in the {\it isospectral description}, where a flag manifold is a submanifold of $\R^{n\times n}$ consisting of symmetric matrices of the form $Y\Lambda Y^\sfT$ with $Y^\sfT Y = \dI_{n_q}$. There is little restriction on the choice of $\Lambda$, except we do not want the eigenvalues to be too close or too different in magnitude. Different choices of $\Lambda$ give us the same manifold topologically. We will use this description only to compare elements in a flag manifold in \cref{prop:small_flag}.

Our main approach is to consider a flag manifold as a quotient manifold of $\St{\R}{n_q}{n}$ by the group of orthogonal diagonal block matrices $\OO(d_1)\times\cdots\times \OO(d_q)$, via the identification $\lb Y\rb\mapsto Y\Lambda Y^{\sfT}$, where $\lb\rb$ denotes the equivalent class in the quotient structure, as above. The action of this group on the Stiefel manifold by right multiplication induces vertical vectors of the form $Y\diag(\rmu_1,\cdots, \rmu_q)$ with $\rmu_i^{\sfT} + \rmu_i = 0$. With the Stiefel metric $\alpha_0\Tr\omega^{\sfT}\omega + (\alpha_1-\alpha_0)\Tr\omega^{\sfT}YY^{\sfT}\omega$ on $\R^{n\times n_q}$, a horizontal vector $\eta$ must satisfy $-\alpha_1\Tr\diag(\rmu_1,\cdots, \rmu_q)Y^{\sfT}\eta = 0$ for all anti-symmetric matrices $\rmu_1,\cdots, \rmu_q$. As each $Y_i^{\sfT}\eta_i$ is already antisymmetric, this implies $Y_i^{\sfT}\eta_i = 0$. Since both $Y$ and $\eta$ are block matrices, $A:=Y^{\sfT}\eta$ is also a block matrix, with block by $A_{[ij]} := Y^{\sfT}_i\eta_j$ ($1\leq i, j\leq p$). Thus, for a horizontal vector $\eta$, besides $A$ being antisymmetric, the diagonal blocks $A_{[ii]}$ must be zero. Fix a matrix $Q$ as an orthogonal basis of the column space of $\eta - YY^\sfT\eta$, we define a block structure on $R=Q^{\sfT}\eta$ by setting $R_{[i]}= Q^{\sfT}\eta_{i}$. For $p=2$, the expression $\eta= YA + QR$ has the following block structure
$$\eta = [\eta_1 | \eta_2] = [Y_1|Y_2|Q]\begin{bmatrix}A_{[11]} & A_{[12]} \\A_{[21]} & A_{[22]} \\R_{[1]} & R_{[2]}\end{bmatrix}
$$
with $A_{[11]} = 0$, $A_{[22]}=0$, $A_{[21]} = - A_{[12]}^\sfT$. To compare, the tangent space of the Stiefel manifold only requires $A = \tidY^{\sfT}\eta$ to be antisymmetric, while a Grassmann horizontal vector requires $A=\tidY^{\sfT}\eta = 0$. For a flag manifold, only the diagonal blocks $A_{[ii]}$ of $A= \tidY^{\sfT}\eta$ are required to be zero.

From the proof of \cref{prop:st_geo}, $A=Y^{\sfT}(t)\dot{Y}(t)$ is constant along any geodesic. Thus, if $Y(0) = \tidY$ and $\dot{Y}(0) = \eta$, the horizontal condition that the diagonal blocks of $Y^{\sfT}(t)\dot{Y}(t)$ vanish is satisfied if $\tidY^{\sfT}\eta$ has vanishing diagonal blocks. Any such horizontal geodesic on $\St{\R}{n_q}{n}$ projects to a geodesic on $\ccF := \Flag(n_1,\cdots,n_q; n)$, and we will represent the exponential map on $\ccF$ by the equivalent class of the exponential map on $\St{\R}{n_q}{n}$.

The Riemannian logarithm between two points $\lb \tidY\rb$ and $\lb \bZ\rb$ in $\ccF$ could be computed by solving for a horizontal tangent vector $\eta$ at $\tidY$ such that $\Exp^{\SSt}_{\tidY}\eta = \bZ U$, for $U\in \OO(d_1)\times\cdots\times \OO(d_q)$, such that $\eta$ has minimal length. Both $\eta$ and $U$ are unknown so this direct approach is difficult. To get an algorithm similar to the Stiefel case, we propose minimizing $\frac{1}{2}||Y(1)\Lambda Y(1)^{\sfT} - \bZ\Lambda \bZ^{\sfT}||_F^2-\Tr\Lambda^2$, as a function of the matrices $A$ and $R$ defined in \cref{eq:st_geo2}. Let $\bd = (d_0, d_1, \cdots, d_q)$ be the vector of the dimension of the blocks, and recall $n_i = \sum_{j=1}^i d_i$ for $i=1,\cdots, q$, set $\AHerm{\sfT, \bd}$ to be the space of antisymmetric matrices $A$ with zero diagonal blocks ($A_{[ii]} = 0$) and $\asym{\sfT, \bd}$ to be the projection from $\R^{n_q\times n_q}$ to $\AHerm{\sfT, \bd}$ by sending a matrix $B$ to $\frac{1}{2}(B-B^{\sfT})$, then set the diagonal blocks to zero. We have
\begin{proposition}\label{prop:small_flag} If $\tidY$ and $Z$ are two points on the Stiefel manifold $\St{\KK}{n_q}{n}$ and let $Q\in \R^{n\times k}$ be an orthogonal basis of $Z - YY^{\sfT}Z$ ($k$ is its rank). Let $(A, R) \in \AHerm{\sfT, \bd}\times \R^{n_q\times k}$, and $Y(t)= Y(A, R; t)$ denotes the geodesic from $\tidY$ with $\dot{Y}(0) = \eta=\tidY A +Q R$ defined by $(A, R)$. Fix an arbitrary diagonal matrix $\Lambda = \diag(\lambda_1 \dI_{d_1},\cdots,\lambda_q\dI_{d_q})$ for a sequence of distinct, non-zero real numbers $\lambda_1,\cdots,\lambda_q$ as in \cref{lem:flag_equiv}. Then a pair $(A, R)$ such that the quotient image $\lb \eta\rb = \lb \tidY A + Q R\rb$ satisfies $\Exp^{\ccF}_{\lb \tidY\rb}\lb \eta\rb =  \lb Z\rb$ for the corresponding equivalent classes $\lb\tidY\rb, \lb Z\rb\in \ccF=\Flag(n_1,\cdots,n_q; n)$ is a minimum of the function:
\begin{equation}\label{eq:flag_log}
  \ttF_{\FFl}(A, R) := \frac{1}{2}||Y(1)\Lambda Y(1)^{\sfT} - Z\Lambda Z^{\sfT}||_F^2 -\Tr\Lambda^2= - \Tr Z^{\sfT}Y(1)\Lambda Y(1)^{\sfT}Z\Lambda
\end{equation}
Its gradient could be evaluated by Fr{\'e}chet derivatives:
\begin{equation}\label{eq:flag_log_gr}
\nabla \ttF_{\FFl}(A, R)= (\nabla_A\ttF_{\FFl}, \nabla_R\ttF_{\FFl}) =((1-2\alpha) \asym{\bd, \sfT}\mathring{L}_{\FFl} + 2\alpha \asym{\bd,\sfT}(\hat{L}_{\FFl, 11}),  \hat{L}_{\FFl, 21} - \hat{L}_{\FFl, 12}^{\sfT})
\end{equation}
where we define $K := [\tidY  Q]$; $H := K^{\sfT}Z\Lambda Z^{\sfT}K$;  $\hat{A} :=  \begin{bmatrix}2\alpha A & -R^{\ft} \\ R & 0\end{bmatrix}$;
$ \begin{bmatrix}U_{11} & U_{12}\\ U_{21} & U_{22}\end{bmatrix}:=\exp\hat{A}$ and
$$\mathring{L}_{\FFl} := \frL_{\exp}(A, 2\Lambda (\exp(1-2\alpha)A)^{\sfT}(U_{11}^{\sfT}HU_{11} + U_{21}^{\sfT}HU_{21}))$$
$$\hat{L}_{\FFl} = \begin{bmatrix}\hat{L}_{\FFl, 11} & \hat{L}_{\FFl, 12}\\  \hat{L}_{\FFl, 21} & \hat{L}_{\FFl, 22}\end{bmatrix} := \frL_{\exp}(\hat{A}, 2\{\exp(1-2\alpha)A\}\Lambda\{\exp(2\alpha-1)A\}H[U_{11}^{\sfT} U_{21}^{\sfT}])$$
where the blocks for $\exp \hat{A}$ and $\hat{L}_{\FFl}$ have the same shapes as those of $\hat{A}$.
\end{proposition}
For the canonical metric, only $\hat{L}_{\FFl} = \frL_{\exp}(\hat{A}, 2\Lambda H[U_{11}^{\sfT} U_{21}^{\sfT}])$ needs to be evaluated.
\begin{proof}
  It is clear $\ttF_{\FFl}+\Tr\Lambda^2$ is bounded below by zero, and from the completeness of flag manifolds, zero is its minimum value. Expanding the first expression for $\ttF_{\FFl}$ in \cref{eq:flag_log} then use $\bZ^{\sfT}\bZ = \dI_{n_q}$ and $Y(1)^{\sfT}Y(1) = \dI_{n_q}$, we get the second expression. Expanding $Y(1)$ using \cref{eq:st_geo3} we have
\begin{equation}\label{eq:ttFccBig}\ttF_{\FFl} = -\Tr Z^{\sfT}K\exp\hat{A} \dI_{n_q+k, n_q}\exp\{(1-2\alpha)A\}\Lambda\exp\{(1-2\alpha)A\}^{\sfT}\dI_{n_q+k, n_q}^{\sfT}(\exp\hat{A})^{\sfT}K^TZ\Lambda \end{equation}
and the expression $K^TZ\Lambda$ from the end could be moved to the front to become $H$. From here, matrix derivative rules imply the directional derivatives of $\ttF_{\FFl}$ in the direction $(\Delta_A, \Delta_R)$ is
\begin{equation}\label{eq:flag_ddr}\begin{gathered}-\Tr H\frL_{\exp}(\hat{A}, \Delta_A) \dI_{n_q+k, n_q}\exp\{(1-2\alpha)A\}\Lambda\exp\{(1-2\alpha)A\}^{\sfT} \dI_{n_q +k, n_q}^{\sfT}(\exp\hat{A})^{\sfT} \\
  -\Tr H\exp\hat{A} \dI_{n_q+k, n_q}D_1 \dI_{n_q+k, n_q}^{\sfT}(\exp\hat{A})^{\sfT}\\
  -\Tr H\exp\hat{A} \dI_{n_q+k, n_q}\exp\{(1-2\alpha)A\}\Lambda\exp\{(1-2\alpha)A\}^{\sfT} \dI_{n_q+k, n_q}^{\sfT}(\frL_{\exp}(\hat{A}, \hat{\Delta}_A)^{\sfT}
  \end{gathered}
\end{equation}  
with $D_1 := D_2 + D_2^{\sfT}$ where $D_2 := \frL_{\exp}((1-2\alpha)A, (1-2\alpha)\Delta_A)\Lambda\exp\{(1-2\alpha)A\}^{\sfT}$ and $\Delta_{\hat{A}} =  \begin{bmatrix}2\alpha \Delta A & -\Delta R^{\ft} \\ \Delta R & 0\end{bmatrix}$.
  The expression on the third line of \cref{eq:flag_ddr} is the same as the first line using transpose-invariance of trace then rearranging the terms. Grouping them then move all factors after $\frL_{\exp}(\hat{A}, \Delta_A)$ of the first line to the front of the trace expression gives us one part of the directional derivative to be
$$\begin{gathered}-2\Tr \dI_{n_q+k, n_q}\exp\{(1-2\alpha)A\}\Lambda\exp\{(1-2\alpha)A\}^{\sfT}\dI_{n_q+k, n_q}^{\sfT}\exp(\hat{A})^{\sfT}H\frL_{\exp}(\hat{A}, \Delta \hat{A})\end{gathered}$$
  which, after expressing $\exp\hat{A}$ in block form and multiplying, then applying \cref{lem:frechet_grad}, simplifies to $-\Tr \hat{L}\Delta_{\hat{A}}$. The middle line of \cref{eq:flag_ddr} could be expressed as
  $$-\Tr H\exp\hat{A} \dI_{n_q+k, n_q} D_1 \dI_{n_q+k, n_q}^{\sfT}\exp\hat{A}^{\sfT} = -\Tr H\begin{bmatrix}U_{11}\\ U_{21}\end{bmatrix} D_1 [U_{11}^{\sfT} U_{21}^{\sfT}]$$
which simplifies to $-\Tr (U_{11}^{\sfT}H U_{11} + U_{21}^{\sfT}H U_{21}) D_1$. Expanding $D_1$ then using transpose invariance of trace, the middle part becomes $-(1-2\alpha)\Tr\mathring{L}\Delta A$. Expressing the matrices in block form and multiply, we are left with the trace
$-\Tr ((1-2\alpha)\mathring{L} + 2\alpha\hat{L}_{\FFl, 11})\Delta A -\Tr \hat{L}_{\FFl, 12} \Delta R + \Tr \hat{L}_{\FFl, 21}\Delta R^{\sfT}$. The gradient formula follows from the fact that for $A\in \AHerm{\sfT, \bd}$ and $B\in\R^{n_q\times n_q}$, we have
$\Tr A^{\sfT}B = \Tr A^{\sfT}\asym{\sfT}B = \Tr A^{\sfT}\asym{\sfT, \bd}B$, the last equality is because the diagonal blocks of $A$ are zeros.
\qed
\end{proof}    
\subsection{Numerical results for flag manifolds}
As before, we minimize $\ttF_{\FFl}(A, R)$ with a trust-region method with numerical Hessian product calculated from $\nabla \ttF_{\FFl}$. As the solver expects an input vector, we vectorize ($A, R)$, then undo vectorization inside the objective function, the gradient, and the Hessian.

Initialization proves to be important when applying the algorithm to flag manifolds. As discussed, we do not optimize for geodesic distance, only for target marching geodesics, relying on the injectivity of the exponential map near the starting point. Empirically, we find the initial point $(A, R) = (0, 0)$ produces minimal length geodesics more consistently. Therefore, we will only use the $A=0, R=0$ initialization.

First, we confirm that for Grassmann manifolds, which are flag manifolds with two blocks ($q=1$), our numerical algorithm closely matches with closed-form formulas, as in \cref{prop:grass}. Testing for various dimensions, each with $1000$ randomly generated points we only find less than $1\%$ of convergence failures, and when we have convergence, the geodesic distance found matches the distance given by the closed-form formula.

Similar to the Stiefel case, we present the results for two flag manifolds, $\Flag(3, 4; 9)$ and $\Flag(20, 35, 40; 1000)$. As before, for each radius and $\alpha$ (each in its own ladder), we generate a random point $\tidY$, then a random vector $\eta_*$ of length equal to the radius and the corresponding ending point $\Exp_{\tidY}\eta_*$. We compute the logarithm by the proposed method, then compare the length of the output logarithm vector with the generating vector $\eta_*$. We find the logarithm vector could be of shorter length than the generating $\eta_*$ in the case $\Flag(3, 4; 9)$, and in general, does not produce a vector of worse length than the generating vector. We observe logarithm for flag manifolds requires more iterations. However, results show convergence is achieved consistently. We analyze the Hessian cut value numerically and confirm it behaves similarly to the Stiefel case.
\begin{table}
  \begin{tabular}{c c c c c c c}
\hline
dist ($\pi$)&$\alpha$&\#Iteration&\#Evals&\%Succ&\%Improve&\%Not worse\\
\hline
0.50 & 0.10 & 5.90 & 136.78 &  100 & 2 & 100\\
0.50 & 0.50 & 5.54 & 109.30 &  100 & 0 & 100\\
0.50 & 1.00 & 5.42 & 97.42 &  100 & 0 & 100\\
0.50 & 1.20 & 5.48 & 105.20 &  100 & 0 & 100\\
0.99 & 0.10 & 10.96 & 382.32 &  96 & 32 & 100\\
0.99 & 0.50 & 9.24 & 275.28 &  100 & 28 & 100\\
0.99 & 1.00 & 7.30 & 195.06 &  100 & 25 & 100\\
0.99 & 1.20 & 7.32 & 182.24 &  100 & 23 & 100\\
1.30 & 0.10 & 9.14 & 257.38 &  100 & 55 & 100\\
1.30 & 0.50 & 7.92 & 206.04 &  100 & 50 & 100\\
1.30 & 1.00 & 7.58 & 180.86 &  100 & 48 & 100\\
1.30 & 1.20 & 7.26 & 174.46 &  100 & 47 & 100\\
\hline
\end{tabular}
  \caption{Logarithm map result for different alpha and distance for $\Flag(3,  4; 9)$ }
\label{tab:num_flag94}
\end{table}

\begin{table}
  \begin{tabular}{c c c c c c c}
\hline
dist ($\pi$)&$\alpha$&\#Iteration&\#Evals&\%Succ&\%Improve&\%Not worse\\
\hline
0.50 & 0.10 & 5.12 & 182.52 &  100 & 0 & 100\\
0.50 & 0.50 & 5.00 & 102.00 &  100 & 0 & 100\\
0.50 & 1.00 & 5.00 & 113.36 &  100 & 0 & 100\\
0.50 & 1.20 & 5.00 & 127.92 &  100 & 0 & 100\\
0.99 & 0.10 & 6.00 & 140.68 &  100 & 0 & 100\\
0.99 & 0.50 & 5.00 & 106.00 &  100 & 0 & 100\\
0.99 & 1.00 & 5.00 & 96.16 &  100 & 0 & 100\\
0.99 & 1.20 & 5.00 & 98.00 &  100 & 0 & 100\\
1.30 & 0.10 & 7.00 & 164.00 &  100 & 0 & 100\\
1.30 & 0.50 & 6.70 & 289.58 &  100 & 0 & 100\\
1.30 & 1.00 & 6.00 & 131.64 &  100 & 0 & 100\\
1.30 & 1.20 & 6.00 & 128.04 &  100 & 0 & 100\\
\hline
\end{tabular}
  \caption{Logarithm map result for different alpha and distance for $\Flag(20,  35,   40; 1000)$ }
\label{tab:num_flag100040}
\end{table}

\subsection{Estimating the Hessian for $\bA=0$ and $\bR=0$ }\label{subsec:flag_A_R}
The role of the eigen matrix $\Lambda$ is revealed below in the Hessian of $\ttF_{\FFl}$ at $A=0, R=0$.
\begin{lemma} If $Q$ satisfies $\tilde{Y}^{\sfT}Q = 0$, let $H_{11} = \tidY^{\sfT}Z\Lambda Z^{\sfT}Y$ be the top left block of $H$ in \cref{prop:small_flag}. The Hessian bilinear form of $\ttF_{\FFl}$ for $A=0, R=0$ at $(\Delta_A, \Delta_R)$ is given by
  \begin{equation}\begin{gathered}\label{eq:Hess_zero_flag}
-\Tr H\{\Delta_{\hat{A}}^2 \dI_{n_q + k, n_q}\Lambda \dI_{n_q + k, n_q}^{\sfT} + \dI_{n_q + k, n_q}\Lambda \dI_{n_q + k, n_q}^{\sfT}\Delta_{\hat{A}}^2 -2\Delta_{\hat{A}}\dI_{n_q + k, n_q}\Lambda\dI_{n_q + k, n_q}^{\sfT}\Delta_{\hat{A}}\}
    \\-2(1-2\alpha)\Tr H\{\Delta_{\hat{A}}\dI_{n_q + k, n_q}(\Delta_A\Lambda - \Lambda\Delta_A)\dI_{n_q + k, n_q}^{\sfT}
+\dI_{n_q + k, n_q}(\Lambda\Delta_A - \Delta_A\Lambda)\dI_{n_q + k, n_q}^{\sfT}\Delta_{\hat{A}}\} \\
- (1-2\alpha)^2 \Tr H_{11}\{(\Delta_A^2\Lambda + \Lambda\Delta_A^2 -2\Delta_A\Lambda\Delta_A)
\}
\end{gathered}
  \end{equation}
At $Z = \tidY$ and $A=0, R=0$, with $\tilde{Y}^{\sfT}Q = 0$, the Hessian bilinear form is reduced to
\begin{equation}\label{eq:hess_zero_flag2}
2\Tr(\Delta_R^{\sfT}\Delta_R\Lambda^2 - \Delta_A^{\sfT}\Lambda\Delta_A\Lambda + \Delta_A^{\sfT}\Delta_A\Lambda^2)  
\end{equation}
The associated Hessian product operator is $(\Delta_A, \Delta_R)\mapsto (\Delta_A\Lambda^2 +\Lambda^2\Delta A - 2\Lambda \Delta_A \Lambda, 2\Delta_R\Lambda^2)$, which 
has eigenvalues $(\lambda_i -\lambda_j)^2$, with the eigenspaces consisting of matrices $A$ with $A_{[ij]}=-A_{[ji]}^{\sfT}\neq 0$, and the remaining blocks of $A$ are zeros, and $2\lambda_i^2$, corresponding to block $R_{[i]}$.
\end{lemma}
In practice, we choose the sequence $\lambda_i$ as consecutive integers centered around $0$, to make the eigenvalues of the Hessian less spread out, and not too small.
\begin{proof}The Hessian is twice the second-order terms of the Taylor series expansion of $\ttF_{\FFl}(A+\Delta_A, R+\Delta_R)$ at $(A, R) = (0, 0)$ from \cref{eq:ttFccBig} and is equal to
  $$\begin{gathered} -\Tr H\{\Delta_{\hat{A}}^2 \dI_{n_q + k, n_q}\Lambda \dI_{n_q + k, n_q}^{\sfT} + \dI_{n_q + k, n_q}\Lambda \dI_{n_q + k, n_q}^{\sfT}\Delta_{\hat{A}}^2 -2\Delta_{\hat{A}}\dI_{n_q + k, n_q}\Lambda\dI_{n_q + k, n_q}^{\sfT}\Delta_{\hat{A}}\}
    \\-2(1-2\alpha)\Tr H\{\Delta_{\hat{A}}\dI_{n_q + k, n_q}(\Delta_A\Lambda - \Lambda\Delta_A)\dI_{n_q + k, n_q}^{\sfT}
+\dI_{n_q + k, n_q}(\Lambda\Delta_A - \Delta_A\Lambda)\dI_{n_q + k, n_q}^{\sfT}\Delta_{\hat{A}}\} \\
- (1-2\alpha)^2 \Tr H\{\dI_{n_q + k, n_q}(\Delta_A^2\Lambda + \Lambda\Delta_A^2 -2\Delta_A\Lambda\Delta_A) \dI_{n_q + k, n_q}^{\sfT}
\}
  \end{gathered}$$
Using the permutation property of trace we can rewrite the last line of the above as
$(1-2\alpha)^2\Tr H_{11}(\Delta_A^2\Lambda + \Lambda \Delta_A^2 - 2\Delta_A\Lambda \Delta_A)$. Together with the other terms we get \cref{eq:Hess_zero_flag}. When $\bZ =\tidY$, $A=0, R=0$, with $\tidY^{\sfT}Q = 0$, $H$ reduces to $\diag(\Lambda, 0_{k, k})$, and the Hessian becomes
$$\begin{gathered}-\Tr\Lambda \{(4\alpha^2\Delta_A^2-\Delta_R^{\sfT}\Delta_R)\Lambda+\Lambda(4\alpha^2\Delta_A^2-\Lambda\Delta_R^{\sfT}\Delta_R) -8\alpha^2\Delta_A\Lambda\Delta_A\}\\
  -2(1-2\alpha)\Tr\Lambda\{2\alpha\Delta_A(\Delta_A\Lambda - \Lambda\Delta_A) + 2\alpha(\Lambda\Delta_A - \Delta_A\Lambda)\Delta_A\} \\
  - (1-2\alpha)^2\Tr\Lambda\{2\Delta_A^2\Lambda - 2\Delta_A\Lambda\Delta_A\}\\
  =\Tr\{-8\alpha^2 -8\alpha(1-2\alpha)-2(1-2\alpha)^2\}\Lambda^2\Delta_A^2   +\Tr\{8\alpha^2 +8\alpha(1-2\alpha) + 2(1-2\alpha)^2\}\Delta_A\Lambda\Delta_A\Lambda\\ +2\Tr\Lambda\Delta_R^{\sfT}\Delta_R\Lambda
\end{gathered} 
$$
which reduces to \cref{eq:hess_zero_flag2}. The operator form and the eigenvalue results follow. \qed
\end{proof}  
\section{Calculation of the Riemannian center of mass}
We apply our proposed logarithm map to compute the Riemannian center of mass \cite{Karcher,Pennec2}. The objective is to minimize $\frac{1}{N}\sum_{i=1}^N d(X, Q_i)^2$ in the variable $X$ in a manifold $\rM$ where $Q_i$'s are given points in $\rM$ ($d$ denotes the geodesic distance). We will assume data is in a sufficiently small cluster and deal with local minima only.

Here, we use the gradient algorithm as considered in \cite{Pennec2}. In our experiment, the data are generated randomly at a distance between $0.4\pi$ and $0.5\pi$ around a center point. The center of mass of $N=50$ points $Q_1,\cdots, Q_N$ is computed by the iteration $X_{i+1} = \Exp_{X_{i}}\frac{1}{N}\Log_{X_{i}}Q_i$. After four iterations, the gradient $\frac{1}{N}\Log_{X_{i}}Q_i$ is less than $5E^{-5}$. The result is summarized in \cref{tab:RCM}. Each gradient iteration performs $N$ computations of Riemannian logarithms by our algorithm. The mean square distance is reduced from $4.0$ to $1.9$ after four iterations for the Stiefel manifold $\St{\R}{50}{100}$, and $3.9$ to $2.0$ for the comparable flag manifold $\Flag(30, 45, 50; 1000)$. The whole calculation finishes after a few minutes for each manifold on a quad-core Xeon(R) E3-1225 @3.20GHz.

\begin{table}
  \centering
  \begin{tabular}{|c| c| c| c| c| c| c|}
\hline
\multicolumn{1}{|c|}{} & \multicolumn{3}{|c|}{Stiefel $\St{\R}{50}{100}$ } & \multicolumn{3}{|c|}{$\Flag(30, 45, 50; 1000)$ }\\
\hline
\backslashbox{iter}{$\alpha$ } &0.10&0.50&1.00&0.10&0.50&1.00\\
\hline
0 & 0.157 & 0.160 & 0.163 & 0.134 & 0.135 & 0.137\\
1 & -1.729 & -1.740 & -1.743 & -1.728 & -1.733 & -1.735\\
2 & -3.614 & -3.639 & -3.650 & -3.587 & -3.600 & -3.606\\
3 & -5.493 & -5.535 & -5.556 & -5.441 & -5.462 & -5.474\\
\hline
\end{tabular}
  \caption{Riemannian center of mass. $\Log_{10}$ of the gradient of mean square distance function for four iterations on 50 data points.}
\label{tab:RCM}
\end{table}
\section{Implementation}The exponential maps are implemented in the $exp$ method of classes {\it RealStiefel, ComplexStiefel, RealFlag, ComplexFlag} in module {\it ManNullRange.manifolds} in \cite{Nguyen2020riemann}. The logarithm maps are implemented by the $log$ method in {\it RealStiefel, RealFlag}. The notebooks {\it StiefelLogarithm, FlagLogarithm} in folder {\it colab} in \cite{Nguyen2020riemann} contain the basic framework for the calculations discussed here.

\section{Conclusion}In this article, we provide efficient, closed-form geodesics formulas of a family of metrics on the Stiefel manifold, extending previously known geodesic formulas often used in optimization. We introduce a novel method to compute a candidate for the logarithm map and show that the method is robust for both Stiefel and flag manifolds. We expect the method, using Fr{\'e}chet derivatives to compute the gradient of the distance function to apply to other manifolds. The method is effective for different parameter values of the metric, thus allowing one to adjust the metric to the problem of interest. We showed it could be used to compute the Riemannian center of mass instead of an approximate geodesic logarithm.

The injectivity radii of Stiefel and flag manifolds are not known for the full family of metrics, although we have provided an analysis based on closed geodesics in the Stiefel case. A more precise estimate of the radii will require an estimate of their sectional curvatures and will be a topic of future research. We hope further works will lead to refinements of the proposed algorithms for the logarithm map. We expect the metrics studied here will find further applications in computer vision, statistics, and other practical problems.
\begin{appendices}
\end{appendices}
%


%
%

\bibliographystyle{spmpsci_unsrt}      
\bibliography{LogStief}   

%
%
\end{document}